\documentclass{amsart}
\usepackage{geometry}                
\usepackage{graphicx}
\usepackage{amssymb}
\usepackage{epstopdf,amsmath}
\newtheorem{prop}{Proposition}[section]
\newtheorem{lemma}[prop]{Lemma}
\newtheorem{cor}[prop]{Corollary}

\newtheorem{theorem}[prop]{Theorem}

\newtheorem{cond}[prop]{Conditions}

\DeclareMathOperator{\Vol}{Vol}
\DeclareMathOperator{\Dim}{Dim}
\DeclareMathOperator{\Poly}{Poly}
\DeclareMathOperator{\Deg}{Deg}

\DeclareMathOperator{\Angle}{Angle}

\DeclareMathOperator{\Area}{Area}
\DeclareMathOperator{\Dist}{Dist}

\DeclareMathOperator{\Br}{Br}
\DeclareMathOperator{\supp}{supp}
\DeclareMathOperator{\Bil}{Bil}
\DeclareMathOperator{\dvol}{dvol}
\DeclareMathOperator{\domega}{d \omega_1 d \omega_2}
\DeclareMathOperator{\neglig}{negligible}
\DeclareMathOperator{\Tan}{Tan}

\newtheorem{definition}[prop]{Definition}

\newcommand{\CC}{\mathbb{C}}
\newcommand{\eps}{\epsilon}

\newcommand{\RR}{\mathbb{R}}
\newcommand{\TT}{\mathbb{T}}
\newcommand{\BrEf}{\Br_\alpha Ef}
\newcommand{\BilEft}{\Bil (E{f}_{j, tang}) }

\title{A restriction estimate using polynomial partitioning}

\author{Larry Guth}

\begin{document}

\begin{abstract} If $S$ is a smooth compact surface in $\mathbb{R}^3$ with strictly positive second fundamental form, and $E_S$ is the corresponding extension operator, then we prove that for all $p > 3.25$, $\| E_S f\|_{L^p(\RR^3)} \le C(p,S) \| f \|_{L^\infty(S)}$.  The proof uses polynomial partitioning arguments from incidence geometry.

\end{abstract}

\maketitle

In this paper we give a small improvement on the 3-dimensional restriction problem using polynomial partitioning.
Suppose that $S \subset \RR^3$ is a smooth surface.  We write $E_S$ for the extension operator.  If $f$ is a function $S \rightarrow \CC$, then

$$ E_S f(x) := \int_S e^{i \omega x} f(\omega) \dvol_S(\omega). $$

\begin{theorem} \label{mainintro} If $S \subset \RR^3$ is a compact $C^\infty$ surface (maybe with boundary) with strictly positive second fundamental form, then for all $p > 3.25$, 

$$ \| E_S f \|_{L^p(\RR^3)} \le C(p, S) \| f \|_{L^\infty(S)} . $$

\end{theorem}

Stein's restriction conjecture \cite{Ste} says that such a bound should hold for all $p > 3$.  An important milestone in the theory was the work of Wolff and Tao (\cite{W1} and \cite{T2}), which proved the estimate above for $p > 10/3$.  This estimate was slightly improved by Bourgain and the author in \cite{BG}, establishing the result for $p > (56/17) = 3.29...$ (see Section 4.8 of \cite{BG}).  Theorem \ref{mainintro} is a further small improvement.  

The main new idea in the current paper is to apply polynomial partitioning to the restriction problem.  
In \cite{D}, Dvir proved the finite field analogue of the Kakeya conjecture by an elegant argument using high degree polynomials.  It remains unclear how much this polynomial method may help to understand the Kakeya conjecture or the restriction conjecture.  I believe that this paper is the first time that the polynomial method has been applied to estimate oscillatory integrals.  Partitioning is an important technique in incidence geometry, introduced by Clarkson, Edelsbrunner, Guibas, Sharir, and Welzl  \cite{CEGSW}.  Polynomial partitioning combines ideas from the partitioning arguments of \cite{CEGSW} and the polynomial arguments of \cite{D}.   It was introduced by Katz and the author in \cite{GK} in our work on the Erd{\H o}s distinct distance problem in incidence geometry.

In the introduction, we will explain how polynomial partioning works in incidence geometry and how to adapt the method to the restriction problem, and we will give a detailed outline of the proof of Theorem \ref{mainintro}.  Before that, we recall background material about incidence geometry and about restriction, and we explain how the two topics are related to each other.  

\subsection{Background on incidence geometry}

Incidence geometry studies the possible intersection patterns of simple geometric objects, such as lines or circles.  Suppose that $\frak L$ is a set of lines in $\RR^n$.  We let $P_r(\frak L)$ be the set of $r$-rich points of $\frak L$: the set of points that lie in at least $r$ lines of $\frak L$.  The most fundamental questions of incidence geometry asks,  ``For given numbers $L$ and $r$, what is the maximum possible number of $r$-rich points that can be formed by a set of $L$ lines?''   Szemer\'edi and Trotter solved this problem up a constant factor in \cite{SzTr}.  Other problems in incidence geometry involve sets of lines with extra conditions, other types of curves, and so on.

Polynomial partitioning is an important recent technique for attacking this type of problem.   Partitioning is a divide-and-conquer approach.  We pick a (non-zero) polynomial $P$, and consider its zero set $Z(P) \subset \RR^n$.  The complement $\RR^n \setminus Z(P)$ is a union of connected components $O_i$, often called cells.  To estimate the size of $P_r(\frak L)$, we can estimate the number of $r$-rich points in each cell $O_i$ and the number of $r$-rich points on the surface $Z(P)$.  One crucial observation is that a line can cross $Z(P)$ at most $\Deg P$ times, and so it can enter at most $1 + \Deg P$ of the cells.  Depending on the choice of $P$, $\RR^n \setminus Z(P)$ can have as many as $\sim (\Deg P)^n$ cells.  If there are $\sim (\Deg P)^n$ cells, then each line enters only a small fraction of the cells.

For this divide-and-conquer approach to be effective, we would like the points of $P_r(\frak L)$ to be evenly divided among the cells $O_i$.  The following partitioning theorem deals with this issue.  The partitioning theorem is a topological result, closely connected to the ham sandwich theorem proven by Stone and Tukey in \cite{ST}.

\begin{theorem} \label{partcomb} (Theorem 4.1 in \cite{GK}) Suppose that $X \subset \RR^n$ is a finite set.  For any $D \ge 1$, there is a polynomial $P$ of degree at most $D$ so that each component of $\RR^n \setminus Z(P)$ contains at most $C_n D^{-n} |X|$ points of $X$.
\end{theorem}

If none of the points of $X$ are in $Z(P)$, then the points have to be quite evenly distributed among the components of $\RR^n \setminus Z(P)$.  We know that there are $\lesssim D^n$ components in total, and each component contains $\lesssim D^{-n} |X|$ points of $X$.  However, it may happen that some or all of the points of $X$ lie in $Z(P)$.  Theorem \ref{partcomb} really gives a kind of dichotomy: either the points cluster on a low degree surface, or else they can be evenly divided by a low degree surface.  

Polynomial partitioning is used in incidence geometry roughly as follows.  If the points of $P_r(\frak L)$ are evenly divided among the cells $O_i$, then we can do a divide-and-conquer argument, estimating the number of $r$-rich points in a typical cell.  For a typical cell $O_i$, the number of lines intersecting $O_i$ is only a small fraction of the $L$ lines.   Then we can estimate the number of $r$-rich points in $O_i$ either directly or by induction.  On the other hand, if the points of $P_r(\frak L)$ cluster on a low-degree surface $Z(P)$, then there is some kind of special structure, and perhaps the original problem reduces to a lower-dimensional problem.

Polynomial partitioning was introduced in \cite{GK}, where it was applied to some problems about lines in $\RR^3$.  In \cite{KMS}, Kaplan, Matousek, and Sharir used polynomial partitioning to give new proofs of some classical results in incidence geometry, including the Szemer\'edi-Trotter theorem.  Polynomial partitioning has been refined and applied to other problems by Solymosi and Tao \cite{SolTao}, Sharir and Solomon \cite{SS}, and others.  

The proof of Theorem \ref{mainintro} uses ideas from these papers, especially the inductive setup introduced in \cite{SolTao}.  In the next subsection, we will give some background on the restriction problem and explain how it connects with incidence geometry. 

\subsection{Background on restriction}

One important example of a positively curved surface $S$ is the truncated paraboloid, defined by $\omega_3 = \omega_1^2 + \omega_2^2, \omega_1^2 + \omega_2^2 \le 1$.  For the rest of the introduction, we focus on this example.  

In \cite{B1}, Bourgain introduced the idea of studying $E_S f$ by breaking it into wave packets.  For a large radius $R$, and for some exponent $p$, we would like to estimate $\int_{B_R} |E_S f|^p$.  We first divide $S$ into caps $\theta$ of radius $\sim R^{-1/2}$.  For each $\theta$, $E_S (f \chi_\theta)$ breaks into pieces supported on tubes.  We let $\TT(\theta)$ be a collection of finitely overlapping tubes covering $B_R$, pointing in the direction of the normal vector to $S$ at $\theta$, with length $\sim R$ and radius roughly $R^{1/2}$.  We can then break $f \chi_{\theta}$ into pieces $f_T$, $T \in \TT(\theta)$ so that $E_S f_T$ is essentially supported on $T$, $f_T$ is essentially supported on $\theta$, and the set of functions $f_T$ are essentially orthogonal.  For each $T \in \TT(\theta)$, $E_S f_T$ on $B_R$ is morally well-approximated by the following model:

$$\textrm{For } x \in B_R,  E_S f_T(x) \textrm{ is approximately } a_T \chi_T e^{i \omega_{\theta} x} ,$$

\noindent where $\omega_{\theta}$ is the center of the cap $\theta$, and $a_T$ is a complex number with $|a_T| \sim R^{-1/2} \| f_T \|_{L^2(\theta)}$.  

Without significant loss of generality, one can imagine that $a_T = 0$ for some tubes $T$ and that $|a_T|$ is constant on all the other tubes.  In this case, $\int_{B_R} |E_S f|^p$ is related to the combinatorics of how the tubes (with $a_T \not= 0$) overlap.  Bourgain \cite{B1} proved combinatorial estimates about overlapping tubes pointing in different directions.  Applying these estimates to the wave packets, he gave new estimates on the restriction problem.  

Wolff (see \cite{W3}) observed that these problems about overlapping tubes have a similar flavor to the problems in incidence geometry we discussed in the last subsection.  He was able to adapt arguments from incidence geometry to prove estimates in analysis.  Using the partitioning argument from \cite{CEGSW}, he proved a Kakeya-type result involving circles \cite{W5} and a local smoothing estimate for the wave equation \cite{W4}.  Following this philosophy, we will adapt the polynomial partitioning approach from incidence geometry to control the wave packets above.

Before turning to polynomial partitioning, we also need to introduce the idea of broad points.
We pick a large constant $K$ and we divide $S$ into $K^2$ caps $\tau$, each of diameter $\sim K^{-1}$, and we write $f_\tau$ for $f \chi_{\tau}$.  For a real number $\alpha \in (0,1)$, we say that $x$ is $\alpha$-broad for $Ef$ if

$$ \max_{\tau} | E f_{\tau} (x)| \le \alpha |Ef(x)|. $$

We define $\Br_\alpha Ef(x)$ to be $|Ef(x)|$ if $x$ is $\alpha$-broad for $Ef$ and zero otherwise.  
From this definition, we see that

\begin{equation} \label{broadnarrow} |E f(x) | \le \max \left( \Br_\alpha Ef (x), \alpha^{-1} \max_\tau |E f_\tau(x)| \right). \end{equation}

The broad contribution is the hardest to estimate, and the second term can be handled by induction as we explain below.  Our strongest result is the following estimate about the broad points:

\begin{theorem} \label{introthm2}  If $S$ is the truncated paraboloid, and $\eps > 0$, then there is a large constant $K = K(\eps)$ so that for any radius $R$

$$ \| \Br_{K^{-\eps}} E_S f \|_{L^{3.25}(B^3_R)} \le C_\eps R^{\eps} \| f \|_{L^2(S)}^{12/13} \| f \|_{L^\infty(S)}^{1/13}. $$

\end{theorem}

We now briefly explain how Theorem \ref{introthm2} implies Theorem \ref{mainintro}.   As an immediate corollary of Theorem \ref{introthm2} we get the estimate:

\begin{equation} \label{broadlinfty} \| \Br_{K^{-\eps}} Ef \|_{L^{3.25}(B^3_R)} \le C_\eps R^{\eps} \| f \|_{L^\infty(S)}. \end{equation}

Following ideas from \cite{BG}, this estimate implies that

\begin{equation} \label{linfty} \| Ef \|_{L^{3.25}(B^3_R)} \le C_\eps R^\eps \| f \|_{L^\infty(S)} \end{equation}

\noindent Here is a quick sketch of the argument.  The idea is to prove Inequality \ref{linfty} by induction on the radius.  By Inequality \ref{broadnarrow}, we see that for any $p$, 

$$ \int_{B_R} |E f|^p \le \int_{B_R} \Br_{K^-\eps} Ef^p + K^{- p \eps} \sum_\tau \int_{B_R} |Ef_\tau|^p. $$

\noindent The broad term on the right-hand side is controlled by Inequality \ref{broadlinfty}.  On the other hand, each integral $\int_{B_R} | E f_\tau |^p$ can be controlled by induction: after a change of variables, it can be controlled using Inequality \ref{linfty} on a smaller ball.   The contributions from the $| E f_\tau |$ terms turn out to be dominated by the contribution from the broad term, and so the induction closes.
This observation is in a similar spirit to the bilinear approach to the restriction problem from \cite{TVV}.  

Finally, by the $\eps$-removal theorem in \cite{T1}, Inequality \ref{linfty} in turn implies Theorem \ref{mainintro} for the paraboloid.  

We also remark that the exponent $3.25$ is the sharp exponent in Theorem \ref{introthm2}, given the right-hand side.  In order to control $L^p$ norms for $p < 3.25$, we would have to weight $\| f\|_\infty$ more and $\| f \|_2$ less.  

\subsection{Examples}

We now give some examples of functions $f$ to illustrate Theorem \ref{introthm2}.  These examples are supposed to give some sense of the theorem, and also to start to illustrate the connection between this theorem and incidence geometry questions.

The first example is a planar example.  In this case, $E_S f$ is essentially supported in a planar slab of dimensions $R^{1/2} \times R \times R$.  There are $\sim R^{1/2}$ caps $\theta \subset S$ for which the normal vector lies within an angle $\sim R^{-1/2}$ of the plane.  For each of these $R^{1/2}$ caps $\theta$, there are $\sim R^{1/2}$ tubes $T \in \TT(\theta)$ that lie in the planar slab.  We pick a number $B$ between 1 and $R^{1/2}$, and for each of the $R^{1/2}$ caps $\theta$, we randomly pick $B$ tubes of $\TT(\theta)$ that lie in our planar slab.  We have now picked $\sim B R^{1/2}$ tubes $T$.  An average point of the planar slab lies in $\sim B$ of our tubes.  Since the tubes were selected randomly, most points of the planar slab lie in $\sim B$ of our tubes.  

For each of our chosen tubes $T$ we choose $f_T$ so that $ |E_S f_T(x) | \gtrsim \chi_T$, and $\| f_T \|_2 \sim R^{1/2}$ and $\| f_T \|_\infty \sim R$.  Now we let $f$ be a sum with random signs: $f = \sum_T \pm f_T$.  Because of the random signs, $|Ef(x)| \gtrsim B^{1/2}$ on most points in the planar slab.  Since the planar slab has volume $\sim R^{5/2}$, $ \| E_S f\|_{L^p(B_R)} \gtrsim B^{1/2} R^{\frac{5}{2p}}. $
Moreover, a typical point lies in $B$ different tubes in random directions (within the plane).  If $B \ge K^{10 \eps}$, then almost every point will be $K^{-\eps}$ broad.  Therefore, we get:

$$ \| Br_{K^{-\eps}} Ef \|_{L^p(B_R)} \gtrsim B^{1/2} R^{\frac{5}{2p}}. $$

On the other hand, we estimate $\| f \|_2$ and $\| f \|_\infty$.  Since the $f_T$ are essentially orthogonal and $f$ is a sum of $B R^{1/2}$ functions $f_T$, and $\| f_T \|_2^2 \sim R$, we get

$$ \| f \|_2 \sim B^{1/2} R^{3/4}. $$

Also, 

$$ \| f \|_\infty \le B \max_T \|f_T \|_\infty \sim B R. $$

The most interesting case for the moment is $B \sim K^{10 \eps}$.  In this case, $B$ is a constant independent of $R$.  
If $ \| Br_{K^{-\eps}} Ef \|_{L^p(B_R)} \le C_\eps R^{\eps} \| f\|_2^{12/13} \| f \|_\infty^{1/13}$, then a direct computation shows that $p \ge 13/4 = 3.25$.  This shows that the exponent $3.25$ in Theorem \ref{introthm2} is sharp, given the right-hand side in the inequality.

It might be possible to get a smaller exponent $p$ by weighting $\| f \|_\infty$ more heavily.  For instance, the following estimate is consistent with the planar example and appears plausible to me: 

$$ \| Br_{K^{-\eps}} Ef \|_{L^3(B_R)} \le C_\eps R^{\eps} \| f\|_2^{2/3} \| f \|_\infty^{1/3}. $$

A second example involves a degree 2 algebraic surface called a regulus.  This example was pointed out to me by Joshua Zahl.  An example of a regulus is the surface $z = xy$.  The key feature of a regulus is that it is doubly ruled, meaning that every point lies in two lines in the surface.  The surface $z=xy$ contains two families of lines:  ``vertical lines'' of the form $x=a$, $z = ay$; and ``horizontal lines'' of the form $y=b$, $z = b x$.  Each point of the regulus lies in one line from each family.  If we want to work in a ball of radius $R$, it is natural to consider a rescaled surface defined by $z/R = (x/R) (y/R)$.  Instead of a planar slab, we consider the $R^{1/2}$-neighborhood of this surface in $B_R$.  This neighborhood contains two families of tubes, corresponding to the horizontal and vertical lines.  We can take $R^{1/2}$ ``horizontal tubes'', and $R^{1/2}$ ``vertical tubes'', all of radius $R^{1/2}$ and length $R$, so that each point lies in at least one horizontal tube and at least one vertical tube.  For each tube $T$, we choose $f_T$ as above so that $|E_S f_T| \gtrsim 1$ on $T$, and so that $\| f_T \|_2 \sim R^{1/2}$ and $\| f_T \|_\infty \sim R$, and we choose $f = \sum_T f_T$.

The computations of $\| E_S f\|_{L^p(B_R)}$ and $\| f \|_2$ and $\| f \|_\infty$ are all the same as in the planar example.  The points in the slab around the regulus are approximately $1/2$-broad.  Because $1/2$ is larger than $K^{-\eps}$, this example is not directly relevant to Theorem \ref{introthm2}, but I think it is morally relevant.  (It is a sharp example for the bilinear restriction estimate in \cite{T1}.)  

These two examples may hint that low-degree polynomial surfaces are relevant to the restriction problem and that if $\| E_S f \|_{L^p(B_R)}$ is large, then there should be a low degree surface where many of the wave packets $E_S f_T$ cluster.  These two low degree examples - planes and reguli - are also relevant in some incidence geometry problems about lines in $\RR^3$.  We consider one such problem in the next subsection and show how to study it using polynomial partitioning.

\subsection{Polynomial partitioning in incidence geometry}

In this section, we demonstrate how polynomial partitioning works in incidence geometry by proving a simple theorem.  This proof will serve as a model for the proof of Theorem \ref{introthm2}.  

Let us first formulate a question about lines in $\RR^3$.  We start with a naive question: how many 2-rich points can be formed by $L$ lines in $\RR^3$?  The answer is $L \choose 2$, which can happen if all the lines lie in a plane.  What if we forbid this simple answer by adding a rule that at most $10$ of the $L$ lines lie in any plane?  Can we still have $\sim L^2$ 2-rich points, or does the number drop off sharply?  The answer is that there can still be $\sim L^2$ 2-rich points.  The second example is that all the lines may lie in a regulus, such as the surface $z= xy$ discussed in the last subsection.  Taking $L/2$ vertical lines and $L/2$ horizontal lines, we get $L^2/4$ 2-rich points.  What if we forbid this example also by adding a rule that not too many lines lie in any plane or degree 2 surface?  We have now arrived at the following question:

If $\frak L$ is a set of $L$ lines in $\RR^3$ with at most $S$ lines in any plane or degree 2 algebraic surface, how big can $|P_2(\frak L)|$ be?

Katz and the author solved this problem in the range $S \ge L^{1/2}$ in \cite{GK}.  (Although it is still open for small values of $S$, such as $S = 10$.)  

\begin{theorem} (See Theorem 2.10 in \cite{GK})  If $\frak L$ is a set of $L$ lines in $\RR^3$ with at most $S$ lines in any plane or degree 2 algebraic surface, then

$$ |P_2(\frak L) | \lesssim S L + L^{3/2}. $$

\end{theorem}

\noindent (If $S \ge L^{1/2}$, then the $SL$ term dominates.  It is possible to get $\sim S L$ 2-rich points by choosing $L / S$ planes, and putting $S$ lines in each plane.  If $S \le L^{1/2}$, then the $L^{3/2}$ dominates.  It is unknown whether this estimate is sharp.)

In order to explain how to use polynomial partitioning, we prove a weak version of this  theorem.  

\begin{theorem} \label{incgeomintro} For any $\eps > 0$, there is a degree $D$ so that the following holds.  

Suppose that $\frak L$ is a set of $L$ lines in $\RR^3$ with at most $S$ lines in any algebraic surface of degree $D$.  Then 

$$ P_2(\frak L) \le C(\eps, S) L^{(3/2) + \eps}.$$  
\end{theorem}

\noindent (This theorem is mostly interesting for small $S$.  In this case, the final estimate is nearly as good as the best known estimate.)

\begin{proof}

The proof goes by induction on $L$.  We apply the polynomial partitioning theorem, Theorem \ref{partcomb}, to the set $P_2(\frak L)$, using polynomials of degree at most $D$.  (We will choose the value of $D = D(\eps)$ below.)  We let $O_i$ be the components of $\RR^n \setminus Z(P)$.  Each $O_i$ contains $\lesssim D^{-3} |P_2(\frak L)|$ points of $P_2(\frak L)$.  By a classical theorem of Milnor, \cite{M}, the number of cells $O_i$ is $\lesssim D^3$.  

If at least half of the points of $P_2(\frak L)$ are in the union of the cells, then we will use induction to study the contribution of each cell.  In this case, there must be $\sim D^3$ cells $O_i$ each containing $\sim D^{-3} |P_2(\frak L)|$ points of $P_2(\frak L)$.  A crucial fact about polynomials that makes them useful in this setting is that a line can intersect $Z(P)$ in at most $D$ points, unless it lies in $Z(P)$.  Therefore, each line of $\frak L$ can enter at most $D+1$ of the cells $O_i$.  Therefore, we can find a cell $O_i$ that intersects $\lesssim D^{-2} L$ lines and contains $\sim D^{-3} |P_2(\frak L)|$ points.  Let $\frak L_i$ be the set of lines of $\frak L$ that enter this cell $O_i$.  Applying induction to bound the 2-rich points of $\frak L_i$, we get the following estimates:

$$ | P_2 (\frak L) | \lesssim D^3 | P_2 (\frak L_i) | \le D^3 C(\eps, S) | \frak L_i| ^{(3/2) + \eps} \lesssim  C(\eps, S) D^3 ( D^{-2} L)^{(3/2) + \eps}.$$

Because of the exponent $(3/2) + \eps$, the total power of $D$ is $D^{-2 \eps}$.  In total we get:

$$ |P_2(\frak L)| \le (C D^{-2 \eps}) C(\eps, S) L^{(3/2) + \eps}, $$

\noindent where $C$ is an absolute constant.  We now choose $D = D(\eps)$ sufficiently large so that $C D^{-2 \eps} < 1$, and the induction closes.

If majority of the points of $P_2(\frak L)$ lie in $Z(P)$, then we estimate $|P_2(\frak L)|$ directly.  
We let $\frak L_Z \subset \frak L$ be the set of lines of $\frak L$ that are contained in $Z(P)$.  Each line of $\frak L \setminus \frak L_Z$ intersects $Z(P)$ in at most $D$ points.  
Therefore, there are at most $DL $ points of $P_2(\frak L) \cap Z(P)$ that involve a line from $\frak L \setminus \frak L_Z$.  Finally we have to estimate $| P_2(\frak L_Z)|$.  By assumption, any algebraic surface of degree at most $D$ contains at most $S$ lines of $\frak L$, and so $| \frak L_Z| \le S$.  Therefore, $| P_2(\frak L_Z)| \le S^2$.  If the majority of the points of $P_2(\frak L)$ lie in $Z(P)$, then we have $|P_2(\frak L)| \le 2 (DL + S^2)$.  By choosing $C(\eps, S)$ sufficiently large, this is at most $C(\eps, S) L^{(3/2) + \eps}$.   \end{proof}
 
To summarize, we estimate $|P_2(\frak L)|$ by breaking it into three contributions: the contributions from the cells $O_i$, the contributions from lines passing through $Z(P)$, and the contribution of lines in $Z(P)$.  We bound the contribution of the cells by induction, using that each cell contributes roughly equally and that each line can enter at most $\sim D$ cells.  
We bound the contribution of lines passing through $Z(P)$ using the fact that each line can only intersect $Z(P)$ in $D$ points.  Finally, we bound the contribution of lines lying in $Z(P)$ using the assumption that not too many lines lie in $Z(P)$.  We also note that this last contribution is a 2-dimensional problem, which makes it simpler than the original problem.  

In the next subsection, we will explain how to apply the polynomial partitioning approach to the restriction problem, and we will again see these three contributions.

\subsection{Polynomial partitioning and the restriction problem}

Now we're ready to start discussing polynomial partitioning and the restriction problem.  We will use the following version of polynomial partitioning, which is also a direct corollary of the Stone-Tukey ham sandwich theorem:

\begin{theorem} \label{partintro} Suppose that $W \ge 0$ is a (non-zero) $L^1$ function on $\RR^n$.  Then for any degree $D \ge 1$, we can find a non-zero polynomial $P$ of degree at most $D$ so that $\RR^n \setminus Z(P)$ is a union of $\sim_n D^n$ disjoint cells $O_i$, and so that all the integrals $\int_{O_i} W$ are equal.  
\end{theorem}

We will give a detailed sketch of the proof of Theorem \ref{introthm2}.   In the introduction, we will write $\Br Ef$ for $\Br_\alpha Ef$ where $\alpha$ is approximately $K^{-\eps}$ but may change a little during the argument.  

We want to estimate the integral $\int_{B_R} \Br Ef^{3.25}$ for a large radius $R$.  We apply the partitioning theorem to the function $\chi_{B_R} \Br Ef^{3.25}$, with a degree $D$ that we will choose below.   By Theorem \ref{partintro}, we can find a polynomial $P$ of degree at most $D$ so that $\RR^n \setminus Z(P)$ is a disjoint union of $\sim D^3$ cells $O_i$, and for each $i$

\begin{equation} \label{equationequi}
\int_{B_R \cap O_i} \Br Ef^{3.25} \sim D^{-3} \int_{B_R} \Br Ef^{3.25}.
\end{equation}

In the combinatorial setting, it was crucial to observe that each line can enter at most $D+1$ cells $O_i$.  In some sense, the tubes $T$ are analogous to lines, but since the tubes have some finite width, it may happen that a tube $T$ enters far more than $D$ cells - a tube $T$ may even enter all of the cells.  Let $W$ be the neighborhood of $Z(P)$ with thickness equal to the radius of a tube $T$.  Define $O_i' := (O_i \cap B_R) \setminus W$.  If a tube $T$ enters $O_i'$, then the central line of $T$ must enter $O_i$.  Therefore, each tube $T$ intersects $O_i'$ for at most $D+1$ values of $i$.

We now break the integral that we care about, $\int_{B_R} \Br Ef^{3.25}$, into pieces coming from the cells $O_i'$ and a piece coming from the cell wall $W$.  (This decomposition is analogous to considering the points of $P_2(\frak L)$ in the cells $O_i$ and the points in $Z(P)$.)  Suppose first that the contribution from the cells dominates the integral.  In this case, there must be $\sim D^3$ cells $O_i'$ so that for each of them

\begin{equation} \label{equationequi'}
\int_{B_R \cap O_i'} \Br Ef^{3.25} \sim D^{-3} \int_{B_R} \Br Ef^{3.25}.
\end{equation}

Since $Ef_T$ decays very sharply outside of $T$, on the set $O_i'$, $Ef$ is essentially equal to the sum of $Ef_T$ over all the $T$ that intersect $O_i'$.  We let $\TT_i$ be the union of all the tubes $T$ (in any $\TT(\theta)$) which intersect $O_i'$, and we define $f_i = \sum_{T \in \TT_i} f_T$.  On $O_i'$, we essentially have $Ef = Ef_i$.  We also essentially have $\Br Ef = \Br Ef_i$.  Now we would like to estimate $\int_{O_i'} \Br Ef_i^{3.25}$ by using induction.

To set up the induction, we have to consider what we know about $f_i$.  Theorem \ref{introthm2} involves $\| f \|_\infty$, but $\| f_i \|_\infty$ is not very well behaved.  We don't have any way to show that $\| f_i \|_\infty$ is significantly smaller than $\| f \|_\infty$, and I think it may even be larger.  Because the functions $f_T$ are essentially orthogonal, we get the following estimate about $f_i$: for each $\theta$, and each $i$, 

\begin{equation} \label{f_Torthog}
\int_\theta |f_i|^2 \lesssim \int_\theta |f|^2.
\end{equation}

\noindent Moreover, because each tube enters $\lesssim D$ cells $O_i'$, the orthogonality of $f_T$ implies that

\begin{equation}
\sum_i \int_S |f_i|^2 \lesssim D \int_S |f|^2. 
\end{equation}

To make the induction work, we need to prove a stronger theorem that involves $\max_\theta \| f \|_{L^2(\theta)}$ instead of $\| f \|_\infty$.  It's convenient to write our inequality in terms of the average of $|f|^2$ over a cap $\theta$, which we write as $\oint_{\theta} |f|^2$.  

\begin{theorem} \label{introbroad12/13}  Let $S$ be the truncated paraboloid.  
For any $\epsilon > 0$, there is a large constant $K = K(\eps)$ so that for every radius $R$ the following holds.  If $f: S \rightarrow \CC$, and for every $R^{-1/2}$-cap $\theta$, 

\begin{equation} \label{avgtheta1}
\oint_\theta |f|^2 \le 1,
\end{equation}

then 

\begin{equation}  \label{12/13equation}
 \int_{B_R} \Br Ef^{3.25} \le C_\eps R^{\eps} \left( \int_S |f|^2 \right)^{(3/2) + \eps}. 
 \end{equation}

\end{theorem} 

This theorem implies Theorem \ref{introthm2} by a direct computation. 
(Recall that in the introduction $\Br Ef$ stands for $\Br_\alpha Ef$ with $\alpha \sim K^{- \eps}$.
Theorem \ref{introbroad12/13} is slightly stronger than Theorem \ref{introthm2}, because it can happen that $\max_\theta \oint_\theta |f|^2 $ is much smaller than $\| f \|_\infty^2$.  In particular, the planar example in the Examples section is sharp for Theorem \ref{introbroad12/13} with any value of $B \ge K^{10 \eps}$.  )
We will see in the proof that the exponent $(3/2) + \eps$ appears here for the same reason that it appeared in the incidence geometry theorem from the last subsection.  Once we have fixed the exponent $(3/2) + \eps$ on the right-hand side, 3.25 is the smallest possible exponent on the left-hand side, because of the planar example.  

We now sketch the proof of Theorem \ref{introbroad12/13}.  To estimate $\int_{B_R} \Br Ef^{3.25}$, we break $B_R$ into cells as above.  Suppose that the integral is dominated by the contribution from the cells.  Then we have $\sim D^3$ cells $O_i'$ so that

$$ \int_{B_R} \Br Ef^{3.25} \lesssim D^3 \int_{O_i'} \Br Ef_i^{3.25}. $$

We can choose one of these cells $O_i'$ so that $\int_S |f_i|^2 \lesssim D^{-2} \int_S |f|^2$, and $\max_\theta \oint_\theta |f_i|^2 \lesssim \max_\theta \oint_\theta |f|^2 \le 1$.  By induction, we can assume that Theorem \ref{introbroad12/13} holds for $f_i$, giving

$$  \int_{B_R} \Br Ef^{3.25} \le C D^3 C_\eps R^{\eps}  \left( D^{-2} \int_S |f|^2 \right)^{(3/2) + \eps} = C D^{-2 \eps} \cdot (\textrm{Right-hand side of equation \ref{12/13equation}}). $$

\noindent We choose $D$ large enough that $C D^{- 2 \eps} \le 1$ and the induction closes.   

Next we consider the case when our integral is dominated by the contribution from $W$, the region near the algebraic surface $Z(P)$.   As in the combinatorial case, there are two kinds of tubes: tubes that pass through $W$ transversally and tubes that lie in $W$.  Roughly, we will show that a tube $T$ can only pass through $W$ transversally in $\lesssim \Poly(D)$ places, and we will use this estimate to bound the transverse tubes using induction.  The tubes that lie in $W$ over a long stretch will be called tangential tubes.  The contribution of the tangential tubes is morally a 2-dimensional problem - similar to the restriction problem in $\RR^2$.  We will bound the tangential contribution by using C\'ordoba's $L^4$ argument from \cite{C}.  

Here is a little bit more detail.  We pick a small parameter $\delta$ so that $R^\delta$ is much bigger than $\Poly(D)$ but still small compared to $R^\eps$.  Now we divide $B_R$ into $\sim R^{3 \delta}$ smaller balls $B_j$ of radius $\sim R^{1 - \delta}$.  For each $j$, we define $\TT_{j, trans}$ to be the set of tubes $T$ that intersect $W \cap B_j$ ``transversally''.  We let $\TT_{j, tang}$ be the set of tubes $T$ that intersect $W \cap B_j$ ``tangentially''.  We will postpone the precise definition to the body of the paper.

To bound the transverse tubes, we first show that any tube $T$ lies in $\TT_{j, trans}$ for at most $\Poly(D)$ different balls $B_j$.  Note that the tube $T$ intersects $\sim R^\delta$ balls $B_j$, and $R^\delta$ is far larger than $\Poly(D)$.  We define $f_{j, trans} = \sum_{T \in \TT_{j, trans}} f_T$.  If the transverse terms dominate, then 

$$ \int_{B_R} \Br Ef^{3.25} \lesssim \sum_j \int_{B_j} \Br Ef_{j, trans}^{3.25}. $$

\noindent Since $B_j$ is smaller than $B_R$, we can assume by induction on the radius that Theorem \ref{introbroad12/13} holds for each integral on the right-hand side.  The average of $|f_{j, trans}|^2$ on a cap of radius $(R^{1 - \delta})^{-1/2}$ is $\lesssim$ the maximum of $\oint_{\theta} |f_{j,trans}|^2$ on a $R^{-1/2}$-cap $\theta$, and $\max_\theta \oint_\theta |f_{j, trans}|^2 \lesssim \max_\theta \oint_\theta |f|^2 \le 1$.  
Moreover, since each tube $T$ lies in only $\Poly(D)$ sets $\TT_{j, trans}$, we get that $\sum_j \int_S |f_{j, trans}|^2 \le \Poly(D) \int_S |f|^2$.  Plugging this in, we get

$$ \int_{B_R} \Br Ef^{3.25} \le \Poly(D) C_\eps (R^{1-\delta})^\eps \left( \int_S |f|^2 \right)^{(3/2) + \eps}  = $$

$$ = \Poly(D) R^{-\delta \eps} \cdot (\textrm{Right-hand side of equation \ref{12/13equation}}). $$

\noindent As long as $\Poly(D) R^{-\delta \eps} \le 1$, the induction closes.  We can assume that $R$ is very large, and we choose $D, \delta$ in such a way that this factor is at most 1.  This method of dealing with the transverse tubes is based on the ``induction-on-scales'' argument from \cite{W1} and \cite{T2}.

Finally, we discuss the contribution of the tangential tubes.  We estimate this contribution directly without using induction.  It might be helpful for the reader to imagine the planar example during this discussion.  In the planar example, the contribution of the tangential tubes would dominate the integral, and the bounds that we prove are all sharp in the planar example.  

One key point is that the set of tangential tubes $\TT_{j, tang}$ cannot contain tubes of $\TT(\theta)$ for every cap $\theta$.  In fact, $\TT_{j, tang}$ can only include contributions from roughly $R^{1/2}$ out of the $R$ caps $\theta$, as in the planar example.  
  
Because we are estimating the broad part of $Ef$, we can reduce the tangential contribution to a bilinear-type estimate.  We can choose $K^{-1}$-separated $K^{-1}$-caps $\tau_1$ and $\tau_2$, and it suffices to bound an integral of the form

\begin{equation} 
 \int_{W \cap B_j} |E f_{\tau_1, j, tang}|^{p/2} |Ef_{\tau_2, j, tang}|^{p/2},
 \end{equation}

\noindent where $f_{\tau_1, j, tang}$ is the sum of $f_T$ where $T \in \TT_{j, tang}$ and $\supp f_T \subset \tau_1$.   The motivation for introducing broad points is to get a bilinear integral at this stage of the argument, instead of the linear integral $\int_{W \cap B_j} | E f_{j, tang}|^p$.  Given our control of $f$, there are much better estimates for the bilinear integral than the linear one.  

We are ultimately interested in $p = 3.25$, but we first prove bounds for $p=2$ and $p=4$ and then interpolate between them.  When $p=2$ the estimate basically boils down to Plancherel.  For $p=4$ we proceed as follows.  

We divide $W \cap B_j$ into cubes $Q$ of side length $\sim R^{1/2}$.  For each cube $Q$, the tubes in $\TT_{j, tang}$ that go through $Q$ lie very close to a plane -- the plane is the tangent plane $T_x Z(P)$ for a point $x \in Z(P)$ near to $Q$.   The angle between the tubes $T$ and the plane is roughly $R^{-1/2}$.  Once we have reduced to the contributions of these coplanar tubes, the problem is essentially 2-dimensional.  As observed in \cite{T2}, the integral $\int_{Q} |E f_{\tau_1, j, tang}|^{2} |Ef_{\tau_2, j, tang}|^{2}$ can be controlled by the $L^4$ argument from \cite{C}.

C\'ordoba's argument gives a square root cancellation estimate.  Recall that $|Ef_T|$ is morally well-modelled by $R^{-1/2} \| f_T\|_2 \chi_T$.  The $L^4$ argument gives the following inequality:

\begin{equation}
\int_{Q} |E f_{\tau_1, j, tang}|^{2} |Ef_{\tau_2, j, tang}|^{2} \lesssim \int_Q \left( \sum_{T_1 \in \TT_{\tau_1, j, tang}} R^{-1} \| f_{T_1} \|^2_2 \chi_{T_1} \right) \left( \sum_{T_2 \in \TT_{\tau_2, j, tang}} R^{-1} \|f_{T_2} \|_2^2 \chi_{T_2} \right) . 
\end{equation}

\noindent Summing over $Q$, it's now straightforward to get a bound for  $\int_{W \cap B_j} |E f_{\tau_1, j, tang}|^{2} |Ef_{\tau_2, j, tang}|^{2}$ and then for $ \int_{W \cap B_j} |E f_{\tau_1, j, tang}|^{p/2} |Ef_{\tau_2, j, tang}|^{p/2}$ with any $2 \le p \le 4$.  At this stage, we can use the fact that $\TT_{j, tang}$ only includes tubes from roughly $R^{1/2}$ caps $\theta$.  The resulting estimates all match the planar example, so they are sharp.

\subsection{Outline of the paper}

In Section 1, we review polynomial partitioning, deducing the partitioning theorem that we need from the Borsuk-Ulam theorem in topology.  In Section 2, we review background related to the restriction problem.  In particular we review wave packet decompositions and parabolic scaling.  In this section, we also review the idea of broad points and explain how to deduce $L^p$ estimates for $Ef$ from $L^p$ estimates for the broad part of $Ef$.  In Section 3 and 4, we prove our main theorem.  Section 3 contains the harmonic analysis part of the argument.  We also need some geometric estimates about the way tubes interact with an algebraic surface.  We prove these estimates in Section 4, using some simple algebraic geometry and differential geometry.

\section{Review of polynomial partitioning}

In this section, we review polynomial partitioning and prove the result that we use.  We will need modifications of the results in the literature, so we give self-contained proofs.  Polynomial partitioning is based on the Stone-Tukey ham sandwich theorem from topology, and we begin by recalling it.  

For any function $f$, we write $Z(f)$ for the zero-set of $f$: $Z(f) := \{ x | f(x) = 0 \}$.  

\begin{theorem} \label{STham} (Stone-Tukey, \cite{ST}) Suppose that $V$ is a vector space of continuous functions on $\RR^n$.  Suppose that for each non-zero element $f \in V$, the set $Z(f) \subset \RR^n$ has measure zero.

Let $W_1, ..., W_N$ be $L^1$-functions on $\RR^n$, and suppose that $N < \Dim V$.  Then there exists a non-zero function $v \in V$ so that for each $W_j$, $j= 1, ..., N$,

$$ \int_{\{ v > 0 \} } W_j = \int_{ \{ v < 0 \} } W_j. $$

\end{theorem}

In our application, $V$ will be the vector space of polynomials on $\RR^n$ of degree at most $D$.  The dimension of this space is ${D + n \choose n} \sim_n D^n$.  It's straightforward to check that for any non-zero polynomial $P$, $Z(P)$ has measure zero.  Therefore, Theorem \ref{STham} has the following corollary:

\begin{cor} \label{polyham} (Polynomial ham sandwich theorem) If $W_1, ..., W_N$ are $L^1$-functions on $\RR^n$, then there exists a non-zero polynomial $P$ of degree $\le C_n N^{1/n}$ so that for each $W_j$,

$$ \int_{\{ P > 0 \} } W_j = \int_{ \{ P < 0 \} } W_j. $$

\end{cor}

The proof of Theorem \ref{STham} is an elegant application of the Borsuk-Ulam theorem, which we now recall.  

\begin{theorem} (Borsuk-Ulam) If $F: S^N \rightarrow \RR^N$ is a continuous function obeying the antipodal condition $F(- v) = - F(v)$, then there exists a $v \in S^N$ with $F(v) = 0$. 
\end{theorem}

The reader can find a proof of the Borsuk-Ulam theorem in \cite{GP} or \cite{Ma}.  

We give the proof of Theorem \ref{STham} using the Borsuk-Ulam theorem:

\begin{proof}
Without loss of generality, we can assume that $\Dim V = N+1$, and we can identify $V$ with $\RR^{N+1}$, so that $S^N \subset V \setminus \{ 0 \}$.  We defining a function $F: V \setminus \{ 0 \} \rightarrow \RR^N$ by setting the $j^{th}$ coordinate to

$$ F_j(v) := \int_{ \{ v > 0 \} } W_j - \int_{ \{ v < 0 \} } W_j . $$

It follows immediately that $F(-v) = - F(v)$.  Moreover, if $F(v) = 0$, then $v$ obeys the conclusion of the ham sandwich theorem.  It is also true that the function $F$ is continuous, which we will check below.  Then the Borsuk-Ulam theorem implies that there exists a $v \in S^N \subset V \setminus \{ 0 \}$ so that $F(v) = 0$.

It just remains to check the continuity of the functions $F_j$ on $V \setminus \{ 0 \}$.  This is a measure theory exercise.  Suppose that $v_k \rightarrow v$ in $V \setminus \{ 0 \}$.  Let $A_k \subset \RR^n$ be the set of points where the sign of $v_k$ is different from the sign of $v$.

$$ | F_j(v_k) - F_j (v) | \le \int_{A_k} |W_j|. $$

We know that the functions $v_k \rightarrow v$ pointwise.  Therefore, $\cap_{k_0} \cup_{k > k_0} A_k \subset v^{-1}(0)$.  By the dominated convergence theorem,

$$ \lim_{k_0 \rightarrow \infty} \int_{\cup_{k \ge k_0} A_k } |W_j| \le \int_{Z(f)} |W_j| = 0. $$

This proves that $\lim_{k \rightarrow \infty} |F_j (v_k) - F_j(v)| = 0$, showing that $F_j$ is continuous on $V \setminus \{ 0 \}$.  \end{proof}

Polynomial partitioning is a corollary of the ham sandwich theorem.  It was proven in \cite{GK} in a discrete setting.  Here we give the same argument in a continuous setting.

\begin{theorem} \label{polypart} Suppose that $W \ge 0$ is a (non-zero) $L^1$ function on $\RR^n$.  Then for each $D$ there a non-zero polynomial $P$ of degree at most $D$ so that $\RR^n \setminus Z(P)$ is a union of $\sim D^n$ disjoint open sets $O_i$, and the integrals $\int_{O_i} W$ are all equal.
\end{theorem}

\begin{proof} Using Corollary \ref{polyham}, we construct a polynomial $P_1$ so that

$$ \int_{\{ P_1 > 0 \} } W = \int_{ \{ P_1 < 0 \} } W = 2^{-1} \int W. $$

Next we let $W_+ = \chi_{ \{ P_1 > 0 \} } W$ and $W_- = \chi_{ \{ P_1 < 0 \} } W$, and we Corollary \ref{polyham} to find a polynomial $P_2$ so that for $j = + $ or $-$, 

$$ \int_{\{ P_2 > 0 \} } W_j = \int_{ \{ P_2 < 0 \} } W_j = 2^{-2} \int W. $$

We have now cut $\RR^n$ into four cells determined by the signs of $P_1,$ and $P_2$.  The integral of $W$ on each cell is equal to $2^{-2} \int W$.  We next construct a polynomial $P_3$ that bisects $W$ restricted to each of these four cells.  

Continuing inductively, we construct polynomials $P_1, ..., P_s$, for a number $s$ that we choose below.  We let $P = \prod P_k$.  The sign conditions of the polynomials cut $\RR^n \setminus Z(P)$ into $2^s$ cells, $O_i$.  The integral of $W$ on each of these cells is equal to $2^{-s} \int W$.  Corollary \ref{polyham} tells us that the degree of $P_k$ is $\lesssim_n 2^{k/n}$.  Therefore, the degree of $P$ is $\le C_n 2^{s/n}$.  Now we choose $s$ so that $C_n 2^{s/n} \in [D/2, D]$, guaranteeing that the degree of $P$ is at most $D$.   The number of cells $O_i$ is $2^s \sim_n D^n$.
\end{proof}

We say that a polynomial $P$ is non-singular if $\nabla P (x) \not= 0$ for each point in $Z(P)$.  If $P$ is non-singular, then it follows that $Z(P)$ is a smooth hypersurface.  For technical reasons, it is helpful in our arguments later to use non-singular polynomials.  We next prove versions of the ham sandwich theorem and the partitioning theorem with non-singular polynomials.  We recall the standard fact that non-singular polynomials are dense.  More precisely, if $\Poly_D(\RR^n)$ denotes the vector space of polynomials on $\RR^n$ of degree at most $D$, then

\begin{lemma} Non-singular polynomials are dense in $\Poly_D(\RR^n)$ for any $D, n$.  Moreover, the singular polynomials have measure zero.
\end{lemma}

\begin{proof} Consider the map $E: \RR^n \times \Poly_D(\RR^n) \rightarrow \RR \times \Poly_D(\RR^n)$, given by $E(x, Q) = (Q(x), Q)$.  The map $E$ is $C^\infty$ smooth, and so by Sard's theorem, the critical values of $E$ have measure zero.  

Suppose that $(h, Q)$ is a regular value of $E$.  Then we claim that $Q - h$ is a non-singular polynomial.  Note that $(Q-h)(x) = 0$ if and only if $(x, Q) \in E^{-1}(h, Q)$.  Since $(h, Q)$ is a regular value, we know that $dE_{x, Q}$ is surjective.  But $dE_{x, Q} = (\nabla Q, id)$, where $id: \Poly_D(\RR^n) \rightarrow \Poly_D(\RR^n)$ is the identity map.  Therefore, if $(Q-h)(x) = 0$, then $\nabla (Q-h)(x) = \nabla Q (x) \not= 0 $.  

We have seen that for almost every $(h, Q)$, $Q-h$ is non-singular.  By Fubini's theorem it follows that the set of singular polynomials has measure zero in $\Poly_D(\RR^n)$, and so the non-singular polynomials are dense. 
\end{proof}

Using the density of non-singular polynomials, we can prove a version of the polynomial ham sandwich theorem with non-singular polynomials, weakening perfect bisections to approximate bisections.

\begin{cor} \label{polyhamns} Suppose that $W_1, ..., W_N \ge 0$ are non-zero functions in $L^1(\RR^n)$.  Then for any $\delta > 0$, there is a non-singular polynomial $P$ so that for each $W_j$

$$(1 - \delta) \int_{ \{ P < 0 \} } W_j \le \int_{ \{ P > 0 \} }W_j \le (1 + \delta) \int_{ \{ P < 0 \}} W_j. $$

\end{cor}

\begin{proof} Let $P_0$ be a non-zero polynomial with $ \int_{\{ P_0 > 0 \} } W_j = \int_{ \{ P_0 < 0 \} } W_j. $  Then let $P_k$ be a sequence of non-singular polynomials approaching $P_0$.  By the continuity argument in the proof of Theorem \ref{STham}, we have $\lim_{k \rightarrow \infty} \int_{ \{ P_k > 0 \} } W_j = \int_{ \{ P > 0 \} } W_j$, and so for large $k$, $P_k$ obeys the desired inequality.
\end{proof}

Finally, using Corollary \ref{polyhamns} in place of Corollary \ref{polyham} in the proof of Theorem \ref{polypart}, we get a partitioning result involving non-singular polynomials.

\begin{cor} \label{partitns} Let $W$ be a non-negative $L^1$ function on $\RR^n$.  Then for any $D$, there is a non-zero polynomial $P$ of degree at most $D$ so that $\RR^n \setminus Z(P)$ is a disjoint union of $\sim D^n$ cells $O_i$, and the integrals $\int_{O_i} W$ agree up to a factor of 2.  Moreover, the polynomial $P$ is a product of non-singular polynomials.  
\end{cor}

\section{Preliminaries}

\subsection{Statement of results}

We will work with surfaces $S$ that are nearly paraboloids.  The basic example is the truncated paraboloid defined by the equation $\omega_3 = \omega_1^2 + \omega_2^2$, $(\omega_1, \omega_2) \in B^2_1(0)$.  The reader may want to focus on this example throughout.
Suppose that $S \subset \RR^3$ is a smooth compact surface given as the graph of a function $h: B^2_1(0) \rightarrow \RR$ which satisfies the following conditions for some large $L$:

\begin{cond} \label{CondSnice}

\begin{enumerate}

\item $0 < 1/2 \le \partial^2 h \le 2$.

\item $0 = h(0) = \partial h(0)$.

\item $h$ is $C^L$, and

\item for $3 \le l \le L$, $\| \partial^l h \|_{C^0} \le 10^{-9}$.    

\end{enumerate}

\end{cond} 

\begin{theorem} \label{mainthm}  For any $\eps > 0$, there is some $L$ so that if $S$ obeys Conditions \ref{CondSnice} with $L$ derivatives, then for any radius $R$, the extension operator $E_S$ obeys the inequality

$$ \| E_S f \|_{L^{3.25}(B_R)} \le C_\eps R^\eps \| f \|_\infty. $$

\end{theorem}

By Tao's $\eps$-removal theorem \cite{T1}, we get the following corollary:

\begin{cor} \label{maincor} If $S$ obeys Conditions \ref{CondSnice}, then for all $p > 3.25$,

$$ \| E_S f \|_{L^p(\RR^3)} \le C(p) \| f \|_\infty. $$

\end{cor}

A little later, at the end of Subsection \ref{parscal}, we will see that 
the case of a general compact surface with positive second fundamental form can be reduced to the case of a surface obeying Conditions \ref{CondSnice}, so that Theorem \ref{mainintro} follows quickly from Corollary \ref{maincor}.  

In coordinates, we have $\omega_3 = h(\omega_1, \omega_2) = h(\vec \omega)$.  We write $\vec \omega \in \RR^2$ for the first two coordinates of $\omega \in \RR^3$.  

\subsection{Broad points}

Let $S$ be as above. We divide $S$ into $\sim K^2$ caps $\tau$ of diameter $\sim K^{-1}$.  Let $f_\tau$ denote the restriction of $f$ to $\tau$.  

For $\alpha \in (0,1)$, we say that $x$ is $\alpha$-broad for $Ef$ if:

$$ \max_{\tau} | Ef_\tau (x)| \le \alpha | E{f}(x)|. $$

\noindent We define $\BrEf (x)$ to be $|Ef(x)|$ if $x$ is $\alpha$-broad for $Ef$ and zero otherwise.  We remark that the definition of $\BrEf(x)$ depends on $K$ and on the choice of the caps $\tau$.  Roughly speaking, if a point $x$ is not broad, then $|Ef(x)|$ is comparable to $|Ef_\tau(x)|$ for some cap $\tau$, and we can deal with these points separately, by some induction on the size of caps.

We will prove the following estimate about $L^p$ norms of the broad part of $Ef$.

\begin{theorem} \label{broad12/13} For any $\epsilon > 0$, there exists $K = K(\eps), L = L(\eps)$ so that if
$S$ obeys conditions \ref{CondSnice} with $L$ derivatives, then for any radius $R$, 

$$ \| \Br_{K^{- \eps}} Ef \|_{L^{3.25} (B_R)} \le C_\eps R^\eps \| f \|_2^{12/13} \| f \|_\infty^{1/13}. $$

\noindent Also, $\lim_{\eps \rightarrow 0} K(\eps) = + \infty$.  

\end{theorem}

We can deduce Theorem \ref{mainthm} from Theorem \ref{broad12/13} using a parabolic scaling argument from \cite{BG} that we explain in the next subsection.

\subsection{Parabolic scaling}  \label{parscal}

Suppose that $B^2_r(\vec \omega_0) \subset B^2_1$.  We let $S_0 \subset S$ be the graph of $h$ over $B^2_r(\omega_0)$.  We can reduce the behavior of the operator $E_{S_0}$ on $B_R$ to the behavior of $E_{S_1}$ on a smaller ball, for a surface $S_1$ which is similar to the original $S$.  If $S$ is a truncated paraboloid $\omega_3 = |\vec \omega|^2$, then $S_1$ will be a truncated paraboloid as well.  This argument involves a change of coordinates which is essentially a parabolic rescaling.

We describe this change of coordinates.  First we define $\tilde h$ to be $h$ minus its first-order Taylor expansion at $\vec \omega_0$:

\begin{equation} \label{tildeh}
 \tilde h(\vec \omega) = h(\vec \omega) - (\vec \omega - \vec \omega_0) \partial h(\vec \omega_0) - h(\vec \omega_0). 
\end{equation}

Next we parametrize $B^2_r(\vec \omega_0)$ by a coordinate $\vec \eta \in B^2(1)$:

$$\vec  \omega =\vec  \omega_0 + r \vec \eta. $$

Now we define the function $h_1$ by

\begin{equation} \label{h_1}
h_1 (\vec \eta) = r^{-2} \tilde h(\vec  \omega) = r^{-2} \tilde h(\vec  \omega_0 + r \vec \eta).
\end{equation}

We let $S_1$ be the graph of $h_1$.  The surface $S_1$ maintains the good properties of $S$.  If $h(\vec \omega) = | \vec \omega|^2$, then $h_1(\vec \eta) = |\vec \eta|^2$.  If $h$ obeys Conditions \ref{CondSnice} with $L$ derivatives, then so does $h_1$.    By equation \ref{tildeh}, we can check that

$$ 0 = \tilde h(\vec  \omega_0) = \partial \tilde h(\vec \omega_0); \partial^2 \tilde h(\vec \omega) = \partial h(\vec \omega). $$

Now using equation \ref{h_1}, we see that $0 = h_1(0) = \partial h_1(0)$.  Also, because of the parabolic rescaling, we have for any indices $i,j$, $\partial^2_{ij} h_1 (\vec \eta) = \partial^2_{ij} h(\vec \omega_0 + r \vec \eta)$.  In particular for all $\vec \eta \in B^2_1$, 

$$ 1/2 \le \partial^2 h_1 \le 2. $$

The function $h_1$ is clearly $C^\infty$ smooth, and another nice feature is that for $l \ge 3$, the $l^{th}$ derivatives of $h_1$ are smaller than for $h$.  In particular, a direct calculation shows that for all $l \ge 2$,

$$ \| \partial^l h_1 \|_{C^0} = r^{l-2} \| \partial^l h \|_{C^0}. $$  

The following lemma connects the behavior of $E_{S_0}$ on $B_R$ to the behavior of $E_{S_1}$ on a smaller ball.

\begin{lemma} \label{parabrescal} Suppose that $h$ obeys Conditions \ref{CondSnice}.  Let $S_1$ be as above: the restriction of the graph of $h$ to a ball of radius $r$.  If $E_{S_1}$ obeys the inequality

$$ \| E_{S_1} g \|_{L^p(B_{10 r R})} \le M \| g \|_{L^\infty(S_1)}, $$

then $E_{S_0}$ obeys the inequality

$$ \| E_{S_0} f \|_{L^p(B_R)} \le C r^{2 - \frac{4}{p}} M  \| f \|_{L^\infty(S_0)}. $$

\end{lemma}

\begin{proof} Let $f \in L^p(S_0)$.  We will express $E_{S_0} f$ using $E_{S_1}$.

$$ |E_{S_0} f(x) | = \left| \int_{S_0} e^{i \omega x} f(\omega) \dvol_{S_0} \right| . $$

Recall that we write $\vec \omega \in \RR^2$ for the first two coordinates of $\omega \in \RR^3$.  Expressing the last integral in these coordinates, we get

$$ = \left| \int_{B^2_r(\vec \omega_0)} e^{i \vec \omega \cdot \vec x} e^{i h(\vec \omega) x_3} f |Jh| d \vec \omega \right|, $$

\noindent where $|Jh_0|$ is the Jacobian $(1 + |\nabla h|^2)^{1/2}$.   Also, we write $\vec x$ for $(x_1, x_2)$.  We rewrite this equation using $\tilde h$ and then using $h_1$. 

$$ = \left| \int_{B^2_r(\vec \omega_0)} e^{i \vec \omega \cdot (\vec x + \partial h(\vec \omega_0) x_3)} e^{i \tilde h(\vec \omega) x_3} f |Jh| d \vec \omega \right| =$$

$$ = \left| \int_{B^2_1} e^{i \vec \eta \cdot r (\vec x + \partial h(\vec \omega_0))} e^{i h_1(\vec \eta) r^2 x_3} f |Jh| r^2 d \vec \eta \right|. $$

This expression is equal to $| E_{S_1} g(\bar x)|$ where

\begin{equation} \label{gdef}
g(\vec \eta) = f(\vec \omega_0 + r \vec \eta) r^2 |J h| |J h_1|^{-1},
\end{equation}

\begin{equation} \label{barxdef}
\bar x = (r x_1 + r \partial_1 h(\omega_0) x_3, r x_2 + r \partial_2 h(\omega_0) x_3, r^2 x_3). 
\end{equation}

Since $\nabla h, \nabla h_1$ vanish at zero, and since $|\nabla^2 h|$ and $|\nabla^2 h_1|$ are at most 2, we know that $|\nabla h|$ and $|\nabla h_1|$ are at most 2 on the unit disk.  Therefore the Jacobian factors $|J h_0|$ and $|Jh_1|$ are $\lesssim 1$.  Therefore, we see from Equation \ref{gdef} that

$$ \| g \|_{L^\infty(S_1)} \lesssim r^{2} \| f \|_{L^\infty(S_0)}. $$

Since $|\partial h(\omega_0)| \le 2$, we see from Equation \ref{barxdef} that if $x \in B_R$, then $\bar x \in B_{10 r R}$.  If we let $\Phi$ be the linear change of coordinates with $\bar x = \Phi (x)$, then the determinant of $\Phi$ is $r^{4}$.  Therefore, we have

$$ \| E_{S_0} f \|_{L^p(B_R)} \le r^{-4/p} \| E_{S_1} g \|_{L^p(B_{ 10 r R})} \le $$

$$ \le r^{-4/p} M \| g \|_{L^\infty(S_1)} \lesssim r^{2 - \frac{4}{p}} M  \| f \|_{L^\infty(S_1)}. $$

\end{proof}

Using parabolic rescaling, we now prove Theorem \ref{mainthm} from Theorem \ref{broad12/13}.

\begin{proof} We will prove the inequality by induction on the radius $R$.  We would like to prove that $\| Ef \|_{L^{3.25}(B_R)} \le \bar C_\eps R^\eps \| f \|_\infty$ for some constant $\bar C_\eps$ indepedent of $R$.  We know that $\| \Br_{K^{-\eps}} Ef\|_{L^{3.25}(B_R)} \le C_\eps R^\eps \| f \|_\infty$.  

We wish to bound $\int_{B_R} | Ef(x)|^{3.25} dx$.  
If $x$ is $K^{- \eps}$-broad, then $|Ef(x)| = \Br Ef$.  If not, then there exists some $K^{-1}$-cap $\tau$ so that $|Ef(x)| \le K^\eps | Ef_\tau(x)|$. Therefore, 

\begin{equation} \label{broad+caps}
\int_{B_R} |Ef|^{3.25} \le \int_{B_R} \Br_{K^{- \eps}} Ef^{3.25} + K^{O(\eps)} \sum_\tau \int_{B_R} |Ef_\tau|^{3.25}. 
\end{equation}

The contribution of the broad term is bounded by Theorem \ref{broad12/13}.  It is at most

$$( C_\eps R^\eps \| f \|_\infty )^{3.25}. $$

We have to prove the same bound for the $Ef_\tau$ terms.  We bound each term using Lemma \ref{parabrescal}.  We let $\tau$ be the graph of $h$ over $B^2_{K^{-1}} (\omega_0)$, and we let $S_1$ be the corresponding surface.  We know that $S_1$ obeys Conditions \ref{CondSnice}.  We can assume that $K$ is large enough so that $10 K^{-1} R < R/2$.  Using induction on $R$ and applying Lemma \ref{parabrescal} with $r = K^{-1}$, we see that

$$ \int_{B_R} |Ef_{\tau}|^{3.25} \le C K^{-2.5} (\bar C_{\eps} R^{\eps} \| f_\tau \|_\infty)^{3.25}. $$

Since there are $\sim K^2$ caps $\tau$, their total contribution to the right-hand side of Equation \ref{broad+caps} is

$$ \le C K^{-(1/2) + O(\eps)} ( \bar C_\eps R^\eps \| f \|_\infty )^{3.25}. $$
 
We also know that $\lim_{\eps \rightarrow \infty} K(\eps) = \infty$.  If $\eps$ is small enough, then $C K^{-(1/2) + O(\eps)} \le 1/100$.  Now choosing $\bar C_\eps = 10 C_\eps$ the induction closes.  \end{proof}

Using parabolic rescaling, we can also deduce Theorem \ref{mainintro} from Corollary \ref{maincor}.
We just sketch the argument, which is standard.  If $S$ is a compact $C^\infty$ surface with strictly positive second fundamental form, then we can divide $S$ into $C(S)$ pieces so that each piece is contained in the graph of a smooth function.  In appropriate orthonormal coordinates, each graph has the form $\omega_3 = h(\vec \omega)$ for $\vec \omega$ contained in a ball of radius $\sim_S 1$.  We can assume that $0 = h(0) = \partial h(0)$.  Because of the positive second fundamental form of $S$, we know that $0 < \lambda \le \partial^2 h \le \Lambda$, and we know that $h$ is $C^\infty$ smooth.  For any $L$, we can do parabolic rescaling with caps of radius $r = r(\lambda, \Lambda, \| h \|_{C^L})$ so that the function $h_1$ will have $| \partial^l h_1 | \le 10^{- 9}$ for all $3 \le l \le L$.  We can do another change of coordinates so that $\partial^2 h_1(0)$ is the identity matrix.  This coordinate change may increase the higher derivatives of $h$, but if we follow by more parabolic rescaling, we are reduced to functions $h$ obeying Conditions \ref{CondSnice}.  The total number of pieces in this decomposition is a constant depending only on $S$.  Applying Theorem \ref{maincor} to each piece and summing, we get Theorem \ref{mainintro}.

\subsection{Wave packet decomposition}

In this subsection, we decompose $Ef$ on $B_R$ into wave packets in a basically standard way.
First we decompose $S$ into $R^{-1/2}$-caps $\theta$.   We let $\omega_\theta$ be a point near the center of $S \cap \theta$, and we let $v_\theta$ denote the unit normal vector to $S$ at $\omega_\theta$.   

Let $\delta > 0$ be a small parameter.  For each cap $\theta$, we let $\mathbb{T}(\theta)$ be a set of cylindrical tubes parallel to $v_\theta$, with radius $R^{(1/2) + \delta}$ and length $\sim R$, covering $B_R$.  We choose the tubes with radius a little bigger than $R^{1/2}$ so that the wave packets decay very sharply outside of the tubes.  For each $\theta$, each point $x \in B_R$ lies in $O(1)$ tubes $T \in \TT(\theta)$.  
We let $\TT = \cup_\theta \TT(\theta)$.  

For any cap $\theta$, we let $3 \theta$ be a larger cap containing $\theta$.  If $\theta$ is the graph of $h$ over a ball $B^2_r(\vec \omega_\theta)$, then we can take $3 \theta$ to be the graph of $\theta$ over $B^2_{3r} (\vec \omega \theta)$.  

If $T$ is a tube in $\TT(\theta)$, we let $v(T) = v_\theta$ be the direction of the tube.

We can now state our result about wave packet decompositions.

\begin{prop} \label{wavepack} Suppose that $S$ obeys Conditions \ref{CondSnice}.  Let $\TT$ be as above, with $\delta > 0$.  Suppose that $R$ is sufficiently large, depending on $\delta$.  If $f$ is a function in $L^2(S)$, then for each $T \in \TT$, we can choose a function $f_T$ so that the following holds:

\begin{enumerate}

\item If $T \in \TT(\theta)$, then $\supp f_T \subset 3 \theta$.

\item If $x \in B_R \setminus T$, then $|E{f_T}(x)| \le R^{-1000} \| f \|_{L^2}$. 

\item For any $x \in B_R$, $| E{f}(x) - \sum_{T \in \TT} E{f_T}(x) | \le R^{-1000} \| f \|_{L^2}$.

\item (essential orthogonality) If $T_1, T_2 \in \TT(\theta)$ and $T_1, T_2$ are disjoint, then $\int f_{T_1} \bar f_{T_2} \le R^{-1000} \int_\theta |f|^2$.

\item $\sum_{T \in \TT(\theta)} \int_S |f_T|^2 \lesssim \int_\theta |f|^2$.  

\end{enumerate}

\end{prop}

\begin{proof} Fix $\theta$.  We define $f_\theta$ to be $f \chi_{\theta}$.  

For each $\theta$ we choose orthonormal coordinates $\omega_1, ..., \omega_3$ so that $5 \theta$ is given by the graph of a function $h$:

$$ \omega_3 = h(\omega_1, \omega_2) = h(\vec \omega). $$

The domain of $h$ is a ball of radius $\sim R^{-1/2}$.  We can choose the coordinates so that $h$ and $\partial h$ vanish at the center of the ball.  Given Conditions \ref{CondSnice}, this function $h$ must obey the following inequalities on the ball:

\begin{equation}
|h| \lesssim R^{-1}; |\nabla h| \lesssim R^{-1/2}; |\nabla^l h| \lesssim_l 1 \textrm{ for all } l \ge 2.
\end{equation}

\noindent We let $(x_1, ..., x_3) = (\vec x, x_3)$ be the dual coordinates to $(\omega_1, ..., \omega_3) = (\vec \omega, \omega_3)$.  

Now we define the tubes of $\TT(\theta)$.  We cover $\RR^{2}$ with finitely overlapping balls $B$ of radius $R^{(1/2) + \delta}$.  We let $T$ be the set of points $x = (\vec x, x_3)$ with $\vec x \in B$.  
We let $\TT(\theta)$ be the set of tubes corresponding to balls $B$ that cover $B^{2}(R)$, and we let $\tilde\TT(\theta)$ be an infinite set of tubes corresponding to balls $B$ that cover $\RR^{2}$.  

We let $\phi_T$ be a partition of unity on $\RR^2$ subordinate to the covering by balls $B$.  
In fact, we make the slightly stronger assumption that the support of $\phi_T(\vec \omega)$ is contained in $(3/4) B$.  We can also think of $\phi_T$ as a partition of unity on $\RR^3$, subordinate to the covering by tubes $T$, where each function $\phi_T(x_1, x_2, x_3)$ is independent of $x_3$.  
We can assume that $| \nabla^l \phi_T | \lesssim_l (R^{(1/2)+ \delta})^{-l}$, and so the Fourier transform obeys the estimate:

$$ | \hat \phi_T (\vec \omega) | \lesssim \Area B \left( 1 + R^{(1/2) + \delta} |\vec \omega| \right)^{10^6}. $$

We let $\psi_\theta$ be a smooth function which is 1 on $2 \theta$ and has support in $3 \theta$.  We can also think of $\psi_\theta(\vec \omega)$ as a function on $\RR^2$.  We can assume that $| \nabla^l \psi_\theta| \lesssim_l R^{l/2}$.

We let $J$ denote the Jacobian factor $( 1 + | \nabla h|^2)^{1/2}$, and we define $F_\theta = J f_\theta$ so that

$$ F_\theta(\vec \omega) \domega = f_\theta(\omega) \dvol_S. $$

We can think of $F_\theta$ either as a function on $\RR^2$ or as a function on $\theta$.  Thinking of $F_\theta(\vec \omega)$ as a function on $\RR^2$, we can define the convolution $\hat \phi_T * F_\theta$.  Now we can define $F_T$ by:

$$ F_T (\vec \omega) := \psi_\theta (\vec \omega) \cdot ( \hat \phi_T * F_\theta) (\vec \omega). $$

We remark that in this formula, the $\psi_\theta$ has a very small effect.  The convolution $\hat \phi_T * F_\theta$ is essentially supported in a small neighborhood of $\theta$, because $\hat \phi_T(\vec \omega)$ decays rapidly for $| \vec \omega| \ge R^{-(1/2) - \delta}$ and $F_\theta$ is supported on $\theta$.  However, $\hat \phi_T * F_\theta$ does have a small tail, which we cut off by multiplying by $\psi_\theta$, so that $F_T$ is supported in $3 \theta$.  

Finally, we define $f_T$ by $F_T = J f_T$ so that

$$ F_T(\vec \omega) \domega = f_T(\omega) \dvol_S. $$

We have now defined $f_T$ and we have to check that it obeys Properties 1-5.

Since $F_T = \psi_{\theta} \cdot ( \hat \phi_T * F_\theta)$, and $\supp \psi_\theta \subset 3 \theta$, it follows that $\supp f_T \subset 3 \theta$, which proves Property 1.

The proof of Property 2 is probably the most important.  Let $T \in \tilde \TT(\theta)$.  
We write $Ef_T(x)$ as
$ \int e^{i \omega x} f_T(\omega) \dvol_S = \int e^{i \vec \omega \cdot \vec x} e^{i h(\vec \omega) x_3} F_T(\vec \omega) \domega$.  Then we plug in that $F_T = \psi_\theta \cdot (\hat \phi_T * F_\theta)$ and group terms to get

\begin{equation} \label{expandingEf_T}
 Ef_T(x) = \int e^{i \vec \omega \cdot \vec x} (e^{i h(\vec \omega) x_3} \psi_\theta) (\hat \phi_T * F_\theta) \domega. 
 \end{equation}

Let $G_{x_3}(\vec \omega) = e^{i h(\vec \omega) x_3} \psi_\theta$.  If we interpret the right-hand side of Equation \ref{expandingEf_T} as an inverse Fourier transform, then intertwining multiplication and convolution, we get:

$$Ef_T(x_1, x_2, x_3) = G_{x_3}^\vee * (\phi_T \cdot \check{F}_\theta). $$

Since $x \in B_R$, $|x_3| \le R$.  It then follows that $| \nabla^l G_{x_3}| \lesssim_l R^{l/2}$, and so

$$ |G_{x_3}^\vee(\vec x)| \lesssim \Area \theta \left(1 + |x| R^{-1/2} \right)^{- 10^6 \delta^{-1} }. $$

Since $x \notin T$, the distance from $x$ to $\supp \phi_T$ is $\ge (1/10) R^{(1/2) + \delta}$.  Finally $| \check{F}_\theta| \lesssim \| f \|_{L^2(\theta)}$.   Plugging these estimates into the convolution, we see that

$$ |Ef_T(x)| \le R^{-10000} \| f \|_{L^2(\theta)}.  $$

This proves Property 2, but for the future we also note a slightly stronger estimate:

\begin{equation} \label{strongprop2}
 |Ef_T(x)| \le R^{-10000} \| f \|_{L^2(\theta)} ( 1 + \Dist(x, T))^{-100}.
\end{equation}

Now we are ready to prove Property 3.  We write $Ef(x)$ as 

$$\sum_\theta E f_\theta(x) = \sum_\theta \int e^{i \omega x} F_\theta (\vec \omega) \domega.$$

Since $\psi_\theta$ is identically 1 on $\supp F_\theta \subset \theta$, we can rewrite this as 

$$ = \sum_\theta \int e^{i \omega x} \psi_\theta F_\theta (\vec \omega) \domega.$$

Now the infinite sum $\psi_\theta \sum_{T \in \tilde \TT(\theta)} \hat \phi_T * F_\theta$ converges to $\psi_\theta F_\theta$ in $L^2$ and hence in $L^1$ since the functions are all supported in $3 \theta$.  Therefore, we can write $Ef(x)$ as a convergent infinite sum:

$$ Ef(x) = \sum_\theta \sum_{T \in \tilde \TT(\theta)} \int e^{i \omega x} \psi_\theta (\hat \phi_T * F_\theta) \domega = \sum_{\theta} \sum_{T \in \tilde \TT(\theta)} Ef_T(x). $$

Finally we want to prune the last sum by including only tubes $T$ in $\TT(\theta)$ -- in other words, only the tubes $T$ that actually intersect $B_R$.  Since $x \in B_R$, the tubes we remove are all disjoint from $x$.  We bound their total contribution using the strong version of Property 2 in equation \ref{strongprop2}.  This proves Property 3.

We prove Property 4 using Plancherel's theorem as follows.  Suppose that $T_1, T_2$ are disjoint tubes in $\TT(\theta)$.  Expanding the definition of $f_{T_1}$ and $f_{T_2}$, we get

\begin{equation} \label{innerprodf_1f_2}
 \int f_{T_1} \overline{ f_{T_2}} \dvol_S = \int J \psi_\theta (\hat \phi_{T_1} * F_\theta) \bar \psi_\theta \overline{(\hat \phi_{T_2} * F_\theta)} \domega = \int \left(J |\psi_\theta|^2 \cdot (\hat \phi_{T_1} * F_\theta) \right) \overline{(\hat \phi_{T_2} * F_\theta)} \domega. 
 \end{equation}

Let $G = J |\psi_\theta|^2$.  Applying Plancherel, our integral is equal to:

\begin{equation} \label{towardsProp4}
\int \left( \check{G} * (\phi_{T_1} \check{F}_\theta) \right) \cdot \overline{ \phi_{T_2} \check{F}_\theta} dx_1 dx_2. 
\end{equation}

Since $T_1$ and $T_2$ are disjoint, $Dist(\supp \phi_{T_1}, \supp \phi_{T_2}) \ge (1/4) R^{(1/2) + \delta}$.  But on the other hand, $G$ obeys $| \nabla^l G | \lesssim_l R^{l/2}$, and so

$$ | \check{G}(\vec x)| \lesssim \Area \theta \left( 1 + |\vec x| R^{-1/2} \right)^{- 10^6 \delta^{-1}}. $$

Also $| \check{F}_\theta(x_1, x_2)| \lesssim \| f \|_{L^2(\theta)}$.  Plugging these bounds into equation \ref{towardsProp4}, we get

$$ \int f_{T_1} \overline{ f_{T_2}} \dvol_S \lesssim R^{-10^5} \| f \|_{L^2(\theta)}^2. $$

This proves Property 4.  

Finally, we turn to Property 5.  Using Equation \ref{innerprodf_1f_2}, we see

$$ \sum_{T \in \TT(\theta)} \int |f_T|^2 \dvol_S = \sum_{T \in \TT(\theta)} \int |\psi_\theta|^2 J |\hat \phi_T * F_\theta|^2 \domega.$$

Since $1 \le J \le 2$, this last integral is

\begin{equation} \label{withpsiintegral}
 \lesssim \sum_{T \in \TT(\theta)} \int |\psi_\theta|^2 |\hat \phi_T * F_\theta|^2 \domega.
 \end{equation}

Since $F_\theta$ is supported in $\theta$ and $\hat \phi_T$ decays rapidly, $\psi_\theta (\hat \phi_T * F_\theta)$ is almost equal to $(\hat \phi_T * F_\theta)$.  In quantitative terms, since $| \hat \phi T (\vec \omega)|$ decays rapidly for $| \vec \omega | \ge R^{-(1/2) - \delta}$, we get

$$ \psi_\theta (\hat \phi_T * F_\theta) (\vec \omega) = (\hat \phi_T * F_\theta) (\vec \omega) + O( R^{-10^5} (1 + | \vec \omega|)^{-10} \| f_\theta \|_2). $$

Using this estimate, we see that line \ref{withpsiintegral} is

$$ \le  \sum_{T \in \TT(\theta)} \int |\hat \phi_T * F_\theta|^2 \domega + O(R^{-10^5} \| f \|_{L^2(\theta)}^2. $$

We can evaluate the last integral by Plancherel, giving

$$ \sum_{T \in \TT(\theta)} \int |\phi_T|^2 |\check{F_\theta}|^2 dx_1 dx_2. $$

But since $\phi_T$ form a partition of unity, $\sum_{T \in \TT(\theta)} |\phi_T|^2 \le 1$, and so the last line is bounded by

$$ \int |\check{F_\theta}|^2 dx_1 dx_2 = \int |F_\theta|^2 \domega \lesssim \int_\theta |f|^2. $$

This proves Property 5.  \end{proof}

We will usually apply Proposition \ref{wavepack} to the functions $f_\tau$.  By Property 1, if $f_\tau$ is supported in $\tau$, then for every $T$, $f_{\tau, T}$ is supported in a $O(R^{-1/2})$ neighborhood of $\tau$.  

Suppose that $\TT_i \subset \TT$ are subsets.  For each $\tau$ and for each subset, we can define a corresponding function $f_{\tau, i}$:

$$ f_{\tau, i} := \sum_{T \in \TT_i} f_{\tau, T}. $$

\begin{lemma} \label{wavepack1} Consider some subsets $\TT_i \subset \TT$ indexed by $i \in I$.  If each tube $T$ belongs to at most $\mu$ of the subsets $\{ \TT_i \}_{i \in I}$, then for every $\theta$,

$$ \sum_{i \in I} \int_{3 \theta} | f_{\tau, i} |^2 \lesssim \mu \int_{10 \theta} |f_\tau|^2. $$

Also, 

$$  \sum_{i \in I} \int_S | f_{\tau, i} |^2 \lesssim \mu \int_S |f_\tau|^2. $$

\end{lemma}

\begin{proof} Each $f_{\tau, i} = \sum_{T \in \TT_i} f_{\tau, T}$.  If $T \in \TT(\theta')$, then $\supp f_{\tau, T} \subset 3 \theta'$.  So in the integral on the left-hand-side, we only need to include the tubes in $\TT(\theta')$ for $O(1)$ caps $\theta'$ each lying in $10 \theta$.  
We define $\TT_i(\theta') := \TT_i \cap \TT(\theta')$, and $f_{\tau, i, \theta'} = \sum_{T \in \TT_{i}(\theta')} f_{\tau, T}$.

$$ \int_{3 \theta} | f_{\tau, i} |^2 \lesssim \sum_{3\theta' \cap 3 \theta \not= \phi} \int |f_{\tau, i, \theta'}|^2 .$$

For each $\theta'$, we expand $f_{\tau, i \theta'}$ to get

$$\sum_i \int |\sum_{T \in \TT_i(\theta')} f_{\tau, T} |^2 = \sum_i \sum_{T_1, T_2 \in \TT_i(\theta')} \int f_{\tau, T_1} \overline{f_{\tau, T_2}}. $$

We control the terms where $T_1$ and $T_2$ are disjoint using Property 4 above.  Each tube $T_1 \in \TT(\theta')$ intersects at most $O(1)$ other tubes $T_2 \in \TT(\theta')$.  Therefore, the last expression is bounded by:

$$ \lesssim \sum_i \left( \sum_{T \in \TT_i(\theta')} \int |f_{\tau, T}|^2 + O( |\TT_i(\theta')|^2 R^{-1000} \| f_{\tau, \theta'} \|_2^2) \right). $$

The big $O$ term contributes at most $|I| R^{-950} \| f_{\tau, \theta'} \|_2^2 \le \mu R^{-900} \| f_{\tau, \theta'} \|_2^2$, which is easily controlled by the right-hand-side.  Using Property 5, the main term is bounded by

$$ \le \mu \sum_{T \in \TT_i(\theta')} \int |f_{\tau, T}|^2 \lesssim \mu \int_{\theta'} |f_\tau|^2. $$

This proves that $\sum_{i \in I} \int_{3 \theta} | f_{\tau, i} |^2 \lesssim \mu \int_{10 \theta} |f_\tau|^2$, giving the first inequality in the conclusion.  Finally, if we sum this inequality over all the caps $\theta \subset S$, we get the second inequality.
\end{proof}

As a special case, applying the lemma above to a single subset $\TT_i \subset \TT$, we get the following:

\begin{lemma} \label{wavepack2}  If $\TT_i \subset \TT$, then for any cap $\theta$, and any $\tau$,

$$ \int_{3 \theta} |f_{\tau, i}|^2 \lesssim \int_{10 \theta} |f_\tau|^2. $$

\end{lemma}

\section{The harmonic analysis part of the proof}

In this section, we give the heart of the proof of Theorem \ref{broad12/13}.  This section contains the proof except for the proofs of some geometric lemmas about how tubes intersect algebraic varieties.  The geometric lemmas have a different flavor, and we prove them in the next section.

\subsection{The inductive setup}

We will prove Theorem \ref{broad12/13} by an inductive argument.  In order to do the induction, we need to set up the Theorem in a slightly more general way.  

Instead of taking $f_\tau$ to be $f$ restricted to $\tau$ and taking the caps $\tau$ disjoint, we need to allow the caps $\tau$ to overlap.  Suppose that each $\tau$ is the graph of $h$ over a ball $B^2(\vec \omega_{\tau}, r)$, and that the union of $\tau$ is $S$.  We consider a decomposition $f = \sum_\tau f_\tau$, where $\supp f_\tau \subset \tau$.  We define $\alpha$-broad as before: $x$ is $\alpha$-broad for $Ef$ if $\max_\tau |Ef_\tau(x)| \le \alpha |Ef(x)|$.

We assume that the centers $\{ \vec \omega_{\tau} \} \subset B^2(1)$ are $K^{-1}$ separated.  We define the multiplicity $\mu$ of the covering by saying that the radius $r$ for each cap $\tau$ lies in the range $[K^{-1}, \mu^{1/2} K^{-1}]$.  Using the radius condition and the separation condition, it follows easily that any point lies in $O(\mu)$ different caps $\tau$.  

\begin{theorem} \label{maintech} For any $\epsilon > 0$, there exists $K, L$ and a small $\delta_{trans} \in (0, \epsilon)$, depending only on $\eps$, so that the following holds.

Suppose that $S$ is the graph of a function $h$ obeying Conditions \ref{CondSnice} for $L$ derivatives.  
Suppose that the caps $\tau$ cover $S$ as described above, with multiplicity at most $\mu$, and suppose that $\alpha \ge K^{-\epsilon}$.

If for any $\tau$ and any $\omega \in S$, 

$$\oint_{B(\omega, R^{-1/2}) \cap S} |f_\tau|^2 \le 1,$$ 

then

$$ \int_{B_R}  \BrEf^{3.25}  \le C_\epsilon R^\epsilon  \left(\sum_\tau \int_S |f_\tau|^2 \right)^{(3/2) + \eps} R^{\delta_{trans} \log (K^\epsilon \alpha \mu)}. $$

\noindent Moreover, $\lim_{\eps \rightarrow 0} K(\eps) = + \infty$.  

\end{theorem}

We can easily recover Theorem \ref{broad12/13} from Theorem \ref{maintech}.  Fix an $\epsilon > 0$.  By scaling $f$,  we can suppose that $\| f\|_\infty = 1$.  We divide $S$ into a disjoint union of $K^{-1}$-caps $\tau$.  The multiplicity of this cover is $\mu \lesssim 1$.  We take $f_\tau = f \chi_\tau$.  So $\sum_\tau \int_S |f_\tau|^2 = \int_S |f|^2$.  Since $\| f \|_\infty = 1$, we see that the average value of $|f_\tau|^2$ on any region is at most 1.  We take $\alpha = K^{-\epsilon}$.  The last factor $R^{\delta_{trans} \log (K^\epsilon \alpha \mu)}$ is $\le R^{C \delta_{trans}} \le R^{O(\epsilon)}$.  Now we can apply Theorem \ref{maintech}, and we see that $\int_{B_R} \BrEf^{3.25} \lesssim C_\eps R^{O(\eps)} (\int_S |f|^2)^{(3/2) + \eps}$.  Since $\| f \|_\infty =1$, this last expression is bounded by $C_\eps R^{O(\eps)} \| f \|_2^{3} \| f \|_{\infty}^{1/4}$. Raising both sides to the power $(3.25)^{-1} = 4/13$, we get $\| \BrEf \|_{L^{3.25}(B_R)} \le C_\eps R^{O(\eps)} \| f \|_2^{12/13} \| f \|_{\infty}^{1/13}$.  Since $\eps > 0$ is arbitrary, we recover Theorem \ref{broad12/13}.

There are several parameters to keep track of.  For reference later, we list them here and say how they are related.  We will take $\delta_{trans} = \epsilon^6$ and $K = e^{\epsilon^{-10}}$.  We also introduce two other small parameters: $\delta = \epsilon^2$.  We will have tubes of thickness $R^{(1/2) + \delta}$.  In the next section, we will choose a degree $D = R^{\delta_{deg}}$ with $\delta_{deg} = \eps^4$.  The key facts about the small parameters are

$$ \delta_{trans} \ll \delta_{deg} \ll \delta \ll \eps. $$

\noindent Also, we need $K$ very large compared to $\delta_{trans}$, so that $R^{\delta_{trans} \log (10^{-6} K^\eps) } \ge R^{1000}$.  

During the proof of Theorem \ref{maintech}, we write $A \lesssim B$ for $A \le C(\eps) B$.  For example, since $K$ is a constant depending on $\eps$, we have $K \lesssim 1$ and $\alpha \gtrsim 1$.  

\subsection{Polynomial partitioning} 
We will prove Theorem \ref{maintech} using polynomial partitioning.  We pick a degree $D = R^{\delta_{deg}}$ with $\delta_{deg} = \eps^4$.  Then we apply polynomial partitioning with this degree to the function $\chi_{B_R} \BrEf^{3.25}$.  Corollary \ref{partitns} tells us that there exists a non-zero polynomial $P$ of degree at most $D$ so that $\RR^n \setminus Z(P)$ is a disjoint union of $\sim D^3$ cells $O_i$, and so that for each $i$, 

$$\int_{O_i \cap B_R} \BrEf ^{3.25} \sim D^{-3} \int_{B_R} \BrEf^{3.25}. $$

\noindent Moreover, we can assume that $P$ is a product of non-singular polynomials.  This is a minor technical point that will help with the proofs of the Lemmas below.

We define $W := N_{R^{(1/2) + \delta}} Z(P)$, and we let $O_i' := (O_i \cap B_R) \setminus W$.   Then we define $\TT_i \subset \TT$ as:

$$ \TT_i := \{ T \in \TT \textrm{ so that }T \cap O_i' \not= \phi \}. $$

We define $f_{\tau, i} = \sum_{T \in \TT_i} f_{\tau, T}$.  We define $f_i = \sum_\tau f_{\tau, i}$.  

We remark that if $T \in \TT_i$, then $T \cap O_i'$ is non-empty, and so the core line of $T$ must intersect $O_i$.  Since a line can cross $Z(P)$ at most $D$ times, we see that each tube $T \in \TT$ intersects at most $D+1$ of the $O_i'$.  We state this estimate as a lemma.

\begin{lemma} Each tube $T \in \TT$ lies in at most $D+1$ of the sets $\TT_i$.  
\end{lemma}

The integral of $\BrEf^{3.25}$ on a cell $O_i'$ will be controlled using induction.  We also have to control the integral of $\BrEf^{3.25}$ on $W$.

We cover $B_R$ with $\sim R^{3 \delta}$ balls $B_j$ of radius $R^{1 - \delta}$.  If $B_j \cap W$ is non-empty, then we note which tubes of $\TT$ are tangent to $Z(P)$ in $B_j$ and which tubes of $\TT$ are transverse to $Z(P)$ in $B_j$.

\begin{definition} \label{deftang}
 $\TT_{j, tang}$ is the set of all $T \in \TT$ obeying the following two conditions:
\begin{itemize}

\item $T \cap W \cap B_j \not= \phi$.

\item If $z$ is any non-singular point of $Z(P)$ lying in $2 B_j \cap 10 T$,  then 

$$\Angle(v(T), T_z Z) \le R^{-(1/2) + 2 \delta}.$$

\end{itemize}
\end{definition}

(Recall that $v(T)$ is the unit vector in the direction of the tube $T$.)

\begin{definition} \label{deftrans}
 $\TT_{j, trans}$ is the set of all $T \in \TT$ obeying the following two conditions:
\begin{itemize}

\item $T \cap W \cap B_j \not= \phi$.

\item There exists a non-singular point $z$ of $Z(P)$ lying in $2 B_j \cap 10 T$,  so that 
$$\Angle(v(T), T_z Z) > R^{-(1/2) + 2 \delta}.$$ 

\end{itemize}
\end{definition}

We claim that any tube $T \in \TT$ that intersects $W \cap B_j$ lies in exactly one of $\TT_{j,tang}$ and $\TT_{j, trans}$.  Looking at the definitions, the only thing that we need to check is that if $T$ intersects $W \cap B_j$, then there is a non-singular point of $Z(P)$ in $10 T \cap 2 B_j$.  We recall that $W$ is the $R^{(1/2)+\delta}$ neighborhood of $Z(P)$, and that $R^{(1/2)+\delta}$ is also the radius of each tube $T$.  Therefore, if $x \in T \cap W \cap B_j$, then there is a point $z \in Z(P)$ with $\Dist(x,z) \le R^{(1/2) + \delta}$.  This point $z$ lies in $10 T \cap 2 B_j$.  Also, since $P$ is a product of non-singular polynomials, the non-singular points are dense in $Z(P)$ and we can assume that $z$ is a non-singular point.  

There are two important geometric lemmas about $\TT_{j, tang}$ and $\TT_{j, trans}$ that we use in our estimates.  We state them here and prove them in the next section.  The proofs use a little algebraic geometry and a little differential geometry.  They have a different flavor from the harmonic analysis arguments we have been discussing, and so we put them in their own section which concentrates on those ideas.  

We begin with an estimate about the transverse tubes. 

\begin{lemma} \label{transbound} Each tube $T \in \TT$ belongs to at most $\Poly(D) = R^{O(\delta_{deg})}$ different sets $\TT_{j, trans}$.
\end{lemma}
 
We remark that a tube $T$ intersects $R^\delta$ different balls $B_j$.  We chose $\delta_{deg} = \epsilon^4$ much smaller than $\delta = \epsilon^2$.  So $T$ belongs to $\TT_{j, trans}$ for only a tiny fraction of these balls.  Using this estimate and induction we can control the contribution from the transverse tubes.  It might also be worth noting the following.  A line can transversely intersect $Z(P)$ in at most $D$ points.  Lemma \ref{transbound} is an analogous estimate with a tube in place of a line.  We get a weaker quantitative bound: polynomial in $D$ instead of linear in $D$.  This is good enough for our purposes, but it would be interesting to understand the worst-case behavior.

Next we give an estimate for the tangential tubes.

\begin{lemma} \label{tangbound} For each $j$, the number of different $\theta$ so that $\TT_{j, tang} \cap \TT(\theta) \not= \phi$ is at most $R^{(1/2) + O(\delta)}$.  
\end{lemma}

There are $\sim R$ different caps $\theta \subset S$.  The lemma says that only on the order of $R^{1/2}$ of these caps can contribute to $\TT_{j, tang}$.  For instance, if $Z(P)$ is a plane, then only the directions tangent to the plane can appear in $\TT_{j, tang}$.  

We let $f_{\tau, j, tang} := \sum_{T \in \TT_{j, tang}} f_{\tau, T}$ and $f_{j, tang} = \sum_\tau f_{\tau, j, tang}$ and similarly for $f_{\tau, j, trans}$ and $f_{j, trans}$.  

\subsection{The inductive step}

In this subsection, we break $\int_{B_R} \BrEf^{3.25}$ into pieces coming from the $f_i$, the $f_{j, trans}$, and the $f_{j, tang}$.  We call these the cellular pieces, the transverse pieces, and the tangential pieces.  We will bound the tangential pieces directly, and we will bound the other pieces by induction.  In this subsection, we explain how to break the integral into pieces, we state the bound for the tangential pieces, and we explain how the induction works.  We will come back to prove the bound for the tangential pieces in the next subsection.

Throughout the arguments, we will assume that $\eps$ is sufficiently small and $R$ is sufficiently large.  

If $x \in O_i'$, then $E{f_\tau}(x)$ is almost equal to $Ef_{\tau, i}(x)$ for each $\tau$.  We also want to think about how the
$\alpha$-broad part of $E{f}(x)$ relates to the $\alpha$-broad part of $E{f_i}(x)$.  

\begin{lemma} \label{declemma1}  If $x \in O_i'$ and $R$ is large enough, then 

$$\Br_\alpha Ef(x) \le 2 \Br_{2 \alpha} E{f}_i(x) + R^{-900} \sum_\tau \| f_\tau \|_2. $$

\end{lemma} 

\begin{proof} By Proposition \ref{wavepack}, we know that

$$ Ef_\tau(x) = \sum_{T \in \TT} Ef_{\tau, T}(x) + O( R^{-1000} \| f_\tau \|_2). $$

If $x \in T$, then $T$ must intersect $O_i'$ so $T \in \TT_i$.  If $x \notin T$, then Proposition \ref{wavepack} gives us the bound $|Ef_{\tau, T}(x)| \le R^{-1000} \| f_\tau \|_2$.  The total contribution of these $T \notin \TT_i$ is small, leaving

\begin{equation} \label{EftalmostEft_i}
Ef_\tau(x) = Ef_{\tau, i}(x) + O(R^{-990} \| f_\tau \|_2).
\end{equation}

Summing over $\tau$, we get

\begin{equation} \label{EfalmostEf_i}
Ef(x) = Ef_i(x) + O(R^{-990} \sum_\tau \| f_\tau \|_2). 
\end{equation}

Now we have to deal with the $\alpha$-broad issue.  We can assume that $|Ef(x)| \ge R^{-900} \sum_\tau \|f_\tau \|_2$ and hence $|Ef_i(x)| \ge (1/2) R^{-900} \sum_\tau \| f_\tau \|_2$.  We can also assume that $x$ is $\alpha$-broad for $Ef$.  Under these assumptions, it remains to show that $x$ is $2\alpha$-broad for $Ef_i$.  In other words, we have to show that for each $\tau$, 

$$ |Ef_{\tau, i}(x)| \le 2 \alpha |E f_i(x)|.  $$

Using Equations \ref{EftalmostEft_i} and \ref{EfalmostEf_i}, we see that

$$ |Ef_{\tau, i}(x)| \le |Ef_\tau(x)| + O(R^{-990} \| f_\tau \|_2) \le \alpha |Ef(x)| + O(R^{-990} \| f_\tau \|_2) \le $$

$$\le \alpha |Ef_i(x)| + O(R^{-990} \sum_\tau \| f_\tau \|_2) \le 2 \alpha |Ef_i(x)|.$$

\end{proof}

If $x \in W \cap B_j$, then the situation is more complicated.  
$E{f}(x)$ is almost equal to $E{f}_{j, trans}(x) + E{f}_{j, tang}(x)$.  
But in order for the $\alpha$-broad parts to behave well, we will need to use not only $E{f}_{j, trans}$ but some other related functions.

Recall that $S$ is divided into $\sim K^2$ caps $\tau$ of diameter $K^{-1}$.  

If $I$ is any subset of these caps, we let $f_{I, j, trans} = \sum_{\tau \in I} f_{\tau, j, trans}$.  The function $f_{I, j, trans}$ comes with a natural decomposition: if $\tau \in I$, we let $f_{\tau, I, j, trans} = f_{\tau, j, trans}$, and if $\tau \notin I$, then $f_{\tau, I, j, trans}  = 0$.  

Eventually we will estimate the terms involving $f_i$ or $f_{j, trans}$ by induction.  On the other hand, we will estimate the terms involving $f_{j, tang}$ by a direct computation.  For this computation, we will use a bilinear version of $f_{j, tang}$ which we now define.  We say that two caps $\tau_1, \tau_2$ are non-adjacent if the distance between them is $\ge K^{-1}$.  

$$ \BilEft := \sum_{\tau_1, \tau_2 \textrm{ non-adjacent}} | E{f}_{\tau_1, j, tang} |^{1/2} | E{f}_{\tau_2, j, tang} |^{1/2}. $$

With these definitions in hand, we can now state our lemma connecting $\BrEf$ with $f_i, f_{j, trans}$, and $f_{j, tang}$.

\begin{lemma} \label{declemma2}

If $x \in B_j \cap W$ and $\alpha \mu \le 10^{-5}$, then

$$ \Br_\alpha |E{f}(x)| \le 2 \left( \sum_I \Br_{2 \alpha} |E{f}_{I, j, trans}(x)| + K^{100} \BilEft(x) + R^{-900} \sum_\tau \| f_\tau \|_2 \right). $$

\end{lemma}

Remark. In Lemma \ref{declemma2}, when we sum over $I$, we are summing over the roughly $2^{K^2}$ subsets of the set of caps $\tau$.  Since $K$ is a constant depending on $\eps$, this large-sounding number will turn out to be minor.  

\begin{proof}

Suppose $x \in B_j \cap W$.  We can assume that $x$ is $\alpha$-broad for $E{f}$ and that $|Ef(x)| \ge R^{-900} \sum_\tau \| f_\tau \|_2$.  

Let $I$ be the set of $K^{-1}$-caps $\tau$ so that $ | Ef_{\tau, j, tang} (x)| \le K^{-100} |Ef (x) |$ .  In other words, $I^c$ is the set of caps $\tau$ so that $ | Ef_{\tau, j, tang} (x)| \ge K^{-100} |Ef (x) |$ .  If $I^c$ contains two non-adjacent caps, then $ |Ef (x) | \le K^{100} \BilEft(x)$, and so the conclusion holds.

If $I^c$ does not contain two non-adjacent caps, then $I^c$ consists of at most $10^4 \mu$ caps, because the centers of the caps are $K^{-1}$ separated, and the radius of each cap is at most $\mu^{1/2} K^{-1}$.  
Since $x$ is $\alpha$-broad for $Ef$, and $\alpha \mu \le 10^{-5}$, we have

$$ \sum_{\tau \in I^c} |Ef_{\tau}(x)| \le 10^4 \mu \alpha |Ef(x)| \le (1/10) | E f(x) |. $$

Therefore, $| Ef_I(x) | \ge (9/10) |Ef(x) |$.  Next, we break up $Ef_I$ into tangential and transverse contributions.

If $T \in \TT$ and $T$ intersects $B_j \cap W$, then $T$ belongs to $\TT_{j, trans}$ or $\TT_{j, tang}$.  On the other hand, if $T$ does not intersect $B_j \cap W$, then $|f_{\tau, T}(x)| = O(R^{-1000} \| f_\tau \|_2)$.  Therefore,
for any cap $\tau$, we have

\begin{equation} \label{eqdeclemmab}
 |Ef_{\tau}(x)| \le |Ef_{\tau, j, trans}(x)| + |Ef_{\tau, j, tang}(x)| + O( R^{-990} \| f_\tau \|_2 ).
\end{equation}

Summing over $\tau \in I$, we see that

$$ |Ef_I(x)| \le |Ef_{I, j, trans}(x)| + \left( \sum_{\tau \in I} |Ef_{\tau, j, tang}(x)| \right) + O( R^{-990} \sum_\tau \| f_\tau \|_2 ). $$

But for each cap $\tau \in I$, $|Ef_{\tau, j, tang}(x)| \le K^{-100} |E f(x)|$, and so $ \sum_{\tau \in I} |Ef_{\tau, j, tang}| \le K^{-98} |Ef(x)|. $  Plugging this in and using that $|Ef_I(x)| \ge (9/10) |Ef(x)|$, we get:

$$ (9/10) |Ef (x)| \le |Ef_{I, j, trans}(x)| + K^{-98} |Ef(x)| + O(R^{-980} \sum_\tau \| f_\tau \|_2). $$

Since $|Ef(x)| \ge R^{-900} \sum_\tau \| f_\tau \|_2$, we see that 

\begin{equation} \label{Ef_Itransbig}
|Ef(x)| \le (3/2) |E f_{I, j, trans}(x)|.  
\end{equation}

In this case, it remains to prove that $x$ is $2 \alpha$-broad for $Ef_{I,j, trans}$.  Given Equation \ref{Ef_Itransbig}, it suffices to prove that for each $\tau \in I$, 

$$ |Ef_{\tau, j, trans}(x)| \le (1.1) \alpha |Ef(x)|. $$

From equation \ref{eqdeclemmab} above, we see that

$$ |Ef_{\tau, j, trans}(x)| \le |Ef_\tau(x)| + |Ef_{j, tang,\tau}(x)| + O(R^{-990} \| f_\tau \|_2). $$

Since $\tau \in I$, $|Ef_{\tau, j, tang}(x)| \le K^{-100} |Ef(x)|$.  Therefore, we have

$$ |Ef_{\tau, j, trans}(x)| \le \alpha |Ef(x)| + K^{-100} |Ef(x)| + O(R^{-990} \| f_\tau \|_2). $$

Because $|Ef(x)| \ge R^{-900} \sum_\tau \|f_\tau \|_2$ and $\alpha \ge K^{-\eps}$, we have

$$ |Ef_{\tau, trans , j}(x)| \le (1.1) \alpha |Ef(x)| . $$

Hence the point $x$ is $2 \alpha$-broad for $Ef_{I, j, trans}$.  \end{proof}

We can now state our estimate for the tangential terms.

\begin{prop} \label{tangtermbound} 

$$\int_{B_j} \BilEft^{3.25} \lesssim R^{O(\delta)} \left(\sum_\tau \int |f_\tau|^2 \right)^{3/2}. $$

\end{prop}

We will prove Proposition \ref{tangtermbound} in the next subsection.  The argument is basically standard.  The proof is important though, and it involves the key moment where we use that the exponent is $3.25$ and not smaller.  

Now we use induction to prove Theorem \ref{maintech}.  We do induction on the radius $R$.  For each radius $R$, we also induct on $\sum_\tau \int |f_\tau|^2$.  As a base of the induction, the theorem is true when $R=1$ or when $\sum_\tau \int |f_\tau|^2 \le R^{-1000}$.  For $R=1$ the theorem is trivial.  If $\sum_\tau \int |f_\tau|^2 \le R^{-1000}$, the theorem follows from observing that $\sup |\Br_\alpha E f| \le ( \sum_\tau \int_S |f_\tau| ) \le C R^{O(\eps)}  ( \sum_\tau \int_S |f_\tau|^2 )^{1/2}$.  Therefore, 

$$ \int_{B_R} |\Br_\alpha E f |^{3.25} \le C R^3 \left( \sum_\tau \int_S |f_\tau| \right)^{3.25} \le 
C R^4 \left( \sum_\tau \int_S |f_\tau|^2 \right)^{(3/2) + (1/8)} \le $$

$$ \le C R^{-100}  \left( \sum_\tau \int_S |f_\tau|^2 \right)^{(3/2) + \eps}. $$

So we can assume Theorem \ref{maintech} holds for radii $\le R/2$ or for functions $g$ with $\sum_\tau \int |g_\tau|^2 \le (1/2) \sum_\tau \int |f_\tau|^2$.  If $\mu \alpha \ge 10^{-6}$, the conclusion of Theorem \ref{maintech} is also trivial, because the factor $R^{\delta_{trans} \log(K^{\epsilon} \alpha \mu)}$ is so large.  We chose $K(\eps) = e^{\eps^{-10}}$ and so the exponent $\eps^6 \log (K^\eps 10^{-6}) \gtrsim \eps^{-4}$.  If $\eps$ is small enough, the factor $R^{\delta_{trans} \log(K^{\epsilon} \alpha \mu)}$ is at least $R^{1000}$, and then the bound is trivially true.  So we can also assume that $\mu \alpha \le 10^{-6}$. 

We decompose our main integral into pieces in the cells and a piece coming from the walls between cells: 

$$\int_{B_R} \BrEf^{3.25} = \sum_i \int_{B_R \cap O_i'} \BrEf^{3.25} + \int_{B_R \cap W} \BrEf^{3.25}.$$

If the cellular term dominates, then we proceed as follows.  Since $\int_{B_R \cap O_i} \BrEf^{3.25}$ is essentially independent of $i$, there must be $\sim D^3$ different cells $O_i'$ so that

\begin{equation} \label{signifi}
 \int_{B_R \cap O_i'} \BrEf^{3.25} \sim D^{-3} \int_{B_R} \BrEf^{3.25}.
 \end{equation}

For each such $i$, applying Lemma \ref{declemma1}, we see that

$$ \int_{B_R} \BrEf^{3.25} \lesssim D^3 \int_{B_R \cap O_i'} \BrEf^{3.25} \lesssim D^3 \int_{B_R} \Br_{2 \alpha} Ef_i^{3.25} + R^{-1000} \sum_\tau \| f_{\tau, i} \|_2^{3.25}. $$

(The last term is a minor error term coming from Lemma \ref{declemma1}.  If that term dominates, then we get the desired bound for $\int_{B_R} \BrEf^{3.25}$ immediately.)

Next we consider $\sum_\tau \int |f_{\tau, i}|^2$.  We noted above that each tube $T$ lies in $\TT_i$ for at most $D+1$ values of $i$.  By Lemma \ref{wavepack1}, we know that $\sum_i \int |f_{\tau, i}|^2 \lesssim D \int |f_\tau|^2$.  Now we can choose a particular $i$ which obeys equation \ref{signifi} and so that

$$ \sum_\tau \int |f_{\tau, i}|^2 \lesssim D^{-2} \sum_\tau \int |f_\tau|^2. $$

We claim that we can apply Theorem \ref{maintech} to $f_i = \sum f_{\tau, i}$.  By Proposition \ref{wavepack}, we know that $\supp f_{\tau, i}$ is in a tiny neighborhood of $\tau$.  Therefore, the new multiplicity is only slightly larger than $\mu$ - it is certainly at most $2 \mu$.  By Lemma \ref{wavepack2}, we know that for any $\omega \in S$, 

$$\oint_{B(\omega, R^{-1/2}) \cap S} |f_{\tau, i}|^2 \lesssim \oint_{B(\omega, 10 R^{-1/2}) \cap S} |f_\tau|^2 \lesssim 1. $$

\noindent Therefore, after multiplying $f_i$ by a constant, it obeys all the assumptions of Theorem \ref{maintech}.  Moreover, $\sum_\tau \int |f_{\tau, i}|^2 \le (1/2) \sum_\tau \int |f_\tau|^2$.  By induction on $\sum_\tau \int |f_\tau|^2$, we can apply Theorem \ref{maintech} to $f_i$.  When we do so, we get the following bound:

$$ \int_{B_R} \BrEf^{3.25} \lesssim D^3 \int_{B_R} \Br_{2 \alpha} Ef_i^{3.25} \lesssim $$

$$ D^3 C_\epsilon R^\eps R^{\delta_{trans} \log (4 \alpha \mu K^\epsilon) } \left( \sum_\tau \int |f_{\tau, i}|^2 \right)^{(3/2) + \eps} .$$

Since $\sum_\tau \int |f_{i,\tau}|^2 \lesssim D^{-2} \sum_\tau \int |f_\tau|^2$, we get all together:

$$ \int_{B_R} \BrEf^{3.25} \le \left(C D^{- 2 \eps} R^{C \delta_{trans} } \right) C_\epsilon R^\eps R^{\delta_{trans} \log (\alpha \mu K^\epsilon) } \left(\sum_\tau \int |f_\tau|^2 \right)^{(3/2) + \eps}. $$

To close the induction, it just suffices to prove that the term in parentheses is $\le 1$.  This term is at most $R^{- \delta_{deg} \epsilon + C \delta_{trans}}$.  Since $\delta_{deg} = \epsilon^4$ and $\delta_{trans} = \epsilon^6$, the exponent of $R$ is negative and the induction closes.

Returning to the decomposition $\int_{B_R} \BrEf^{3.25} = \sum_i \int_{B_R \cap O_i'} \BrEf^{3.25} + \int_{B_R \cap W} \BrEf^{3.25}$, let us now suppose that the contribution from the cell walls dominates.  By Lemma \ref{declemma2}, we now have

$$ \int_{B_R} \BrEf^{3.25} \lesssim $$
$$\sum_{j, I} \int_{B_j} \Br_{2 \alpha} Ef_{I, j, trans}^{3.25} + \sum_j K^{100} \int_{B_j} \BilEft^{3.25} + O(R^{-1000} \sum_\tau \| f_\tau \|_2^{3.25}). $$

If the final $O$-term dominates, then the conclusion holds trivially, using the fact that $\sum_\tau \| f_\tau \|_2^2 \lesssim 1$.  
By Proposition \ref{tangtermbound}, we know that the tangential term is bounded by $R^{O(\delta)} (\sum_\tau \int |f_\tau|^2)^{3/2} \le R^\epsilon (\sum_\tau \int |f_\tau|^2)^{3/2}$.  So if the tangential term dominates we are also done.  Therefore, we are left with the case where

\begin{equation} \label{transversecase}
 \int_{B_R} \BrEf^{3.25} \lesssim \sum_{j, I} \int_{B_j} \Br_{2 \alpha} E f_{j, trans, I}^{3.25}.
 \end{equation}

We claim that we can apply Theorem \ref{maintech} to each integral on the right-hand-side.  The ball $B_j$ has radius $R^{1-\delta}$, so by induction on the radius Theorem \ref{maintech} applies.  We have to check that $f_{j, trans,I}$ satisfies the hypotheses.  By Proposition \ref{wavepack}, $\supp f_{\tau, I, j, trans}$ lies in a small neighborhood of $\tau$ - a slightly larger cap.  As above, the multiplicity of the new covering with slightly larger caps is at most $2 \mu$.   By Lemma \ref{wavepack2}, we have for any $\omega \in S$, 

$$ \oint_{B(\omega, R^{-1/2}) \cap S} |f_{j, trans, I,\tau}|^2 \lesssim \oint_{B(\omega, 10 R^{-1/2}) \cap S} |f_\tau|^2 \lesssim 1. $$

\noindent  Therefore, we may apply Theorem \ref{maintech} to each of the integrals on the right-hand side of Equation \ref{transversecase}.  We get the following upper bound: 

$$ \int_{B_j}  \Br_{2 \alpha} Ef_{j, trans, I}^{3.25} \lesssim C_\eps R^{(1 - \delta) \eps} R^{\delta_{trans} \log (4 \alpha \mu K^\epsilon)} (\sum_\tau \int |f_{\tau, j, trans}|^2)^{(3/2) + \eps}.$$

To bound $\int_{B_R} \BrEf^{3.25}$, we have to sum over all $j, I$.  Now the crucial point is Lemma \ref{transbound}, which tells us that a given tube $T$ lies in $\TT_{j, trans}$ for at most $\Poly(D)$ values of $j$.  (The number of different values of $I$ is only a constant depending on $\epsilon$.)  Therefore, by Lemma \ref{wavepack1}, 

$$ \sum_{j} \int |f_{\tau, j, trans}|^2 \lesssim \Poly(D) \sum_\tau \int |f_\tau|^2, $$

and hence

$$ \sum_{j, I} (\sum_{\tau \in I} \int |f_{\tau, j, trans}|^2)^{(3/2) + \eps} \lesssim \Poly(D) (\sum_\tau \int |f_\tau|^2)^{(3/2) + \eps}. $$

Summing over $j, I$ and plugging this in, we get the following bound:

$$ \int_{B_R} \BrEf^{3.25} \le \Poly(D) C_\eps R^{(1 - \delta) \eps} R^{\delta_{trans} \log (4 \alpha \mu K^{\epsilon})} (\sum_\tau \int |f_\tau|^2 )^{(3/2) + \eps} = $$

$$ = \left( C \Poly(D) R^{-\delta \eps} R^{C \delta_{trans}} \right) C_\eps R^{\eps} R^{\delta_{trans} \log( \alpha \mu K^\eps)} (\sum_\tau \int |f_\tau|^2 )^{(3/2) + \eps}. $$

To close the induction, we just have to check that the term in parentheses is less than 1.  For sufficiently large $R$, this term is at most $R^{C \delta_{deg} - \delta \eps + C \delta_{trans} }$.  Since $\delta = \epsilon^2$, $\delta_{deg} = \epsilon^4$, and $\delta_{trans} = \epsilon^6$, the exponent of $R$ is negative and the induction closes.

We have now finished carrying out the induction.  It only remains to prove the bound for the tangential terms in Proposition \ref{tangtermbound}.  

\subsection{The estimate for the tangential terms}

In this subsection, we prove Proposition \ref{tangtermbound}.  In other words, we have to prove the following estimate:

$$\int_{B_j \cap W} \BilEft^{3.25} \lesssim R^{O(\delta)} \left(\sum_\tau \int |f_\tau|^2 \right)^{3/2}. $$

Cover $B_j \cap W$ with cubes $Q$ of side length $R^{1/2}$.  For each cube $Q$, we let $\TT_{j, tang, Q}$ be the set of tubes in $\TT_{j, tang}$ that intersect $Q$.  On $Q$, we have

$$ Ef_{\tau, j, tang} = \sum_{T \in \TT_{j, tang, Q}} Ef_{\tau, T} + O(R^{-990} \| f_{\tau} \|_2). $$

The terms of the form $O(R^{-990} \| f_\tau \|_2)$ are always negligible in our calculations, and in this subsection, we will abbreviate them by writing

\begin{equation} \label{Ef_tau}
 Ef_{\tau, j, tang} = \sum_{T \in \TT_{j, tang, Q}} Ef_{\tau, T} + \neglig.
  \end{equation}

Because of the definition of $\TT_{j, tang}$, Definition \ref{deftang}, we claim that all the tubes in $\TT_{j, tang, Q}$ are nearly coplanar.  Since $Q \cap W$ is non-empty, there must be a point $z \in Z(P)$ in the $R^{(1/2)+\delta}$-neighborhood of $Q$.  For any $T \in \TT_{j,tang, Q}$, $z \in 10T \cap 2B_j \cap Z(P)$.  Also, since $P$ is a product of non-singular polynomials, the non-singular points are dense in $Z(P)$, and so we can assume that $z$ is non-singular.  Now by Definition \ref{deftang}, the angle between $v(T)$ and $T_z Z(P)$ is $\le R^{-(1/2) + 2 \delta} \le R^{-(1/2) + O(\delta)}$.

Using this observation and the C\'ordoba $L^4$ argument, we get a bilinear estimate on $Q$:

\begin{lemma} \label{locbilinear} If $\tau_1$ and $\tau_2$ are non-adjacent caps, then

$$ \int_{Q} |Ef_{\tau_1, j, tang}|^2 |Ef_{\tau_2, j, tang}|^2 \lesssim R^{O(\delta)} R^{-1/2} (\sum_{T_1 \in \TT_{j, tang, Q}} \| f_{\tau_1, T_1} \|_2^2 ) (\sum_{T_2 \in \TT_{j, tang, Q}} \| f_{\tau_2, T_2} \|_2^2 ) + \neglig . $$

\end{lemma}

\begin{proof}

On $Q$, we have 

$$Ef_{\tau, j, tang} = \sum_{T \in \TT_{j, tang, Q}} Ef_{\tau, T} + \neglig. $$

We let $\eta_Q$ be a smooth bump function which is equal to 1 on $Q$ and with support in $10 Q$.  (We can assume that $| \hat \eta_Q (\omega) | \lesssim \Vol(Q) (1 + |\omega| R^{1/2})^{10^6 \delta^{-1}}. $)  Now we can bound

$$ \int_Q |Ef_{\tau_1, j, tang}|^2 |Ef_{\tau_2, j, tang}|^2 \le
\sum_{T_1, \bar T_1, T_2, \bar T_2 \in \TT_{j, tang, Q}} \int \eta_Q Ef_{\tau_1, T_1} \overline{Ef_{\tau_1, \bar T_1}} Ef_{\tau_2, T_2} \overline{Ef_{\tau_2, \bar T_2}} + \neglig. $$

Each of the summands on the right-hand side we can evaluate with Plancherel, giving

\begin{equation} \label{cordobasum}
 \sum_{T_1, \bar T_1, T_2, \bar T_2 \in \TT_{Q,j, tang}} \int_{\RR^3} (\hat \eta_Q * f_{\tau_1, T_1} \dvol_S * f_{\tau_2, T_2} \dvol_S) \overline{(f_{\tau_1, \bar T_1} \dvol_S * f_{\tau_2, \bar T_2} \dvol_S )} .
 \end{equation}

Only very few of these terms are significant.  For each tube $T$, let $\theta(T)$ denote the cap $\theta$ so that $T \in \TT(\theta)$, and let $\omega(T)$ be the center of $\theta(T)$.  The measure $f_{\tau, T} \dvol_S$ is supported on $3 \theta(T)$, and so the support lies in the $O(R^{1/2})$-neighborhood of $\omega(T)$.  Because of the rapid decay of $\hat \eta_Q$, a term in the sum above is negligible unless

\begin{equation} \label{T_1+T_2=bar}
\omega(T_1) + \omega(T_2) = \omega(\bar T_1) + \omega (\bar T_2) + O(R^{- (1/2) + \delta}). 
\end{equation}

Next we claim that equation \ref{T_1+T_2=bar} forces $\omega(T_1)$ to be $O(R^{-(1/2) + \delta})$ close to $\omega(\bar T_1)$, and the same for $\omega(T_2)$ and $\omega(\bar T_2)$.  We know that $v(T_i)$ and $v(\bar T_i)$ all lie in a common plane $\pi(Q)$.  Recall that $v(T_i)$ is essentially the unit normal vector to $S$ at $\omega(T_i)$.  Therefore, at each point $\vec \omega(T_i), \vec \omega(\bar T_i) \in B^2(1)$, $\nabla h$ satisfies a linear equation:

\begin{equation}
m \cdot \nabla h(\vec \omega) + b = 0,
\end{equation}

\noindent for a vector $m \in \RR^2$ with $|m| \le 1$, and a number $b$ with $|b| \lesssim 1$.

This equation defines a curve in $B^2(1)$.  If $h$ were exactly quadratic, then this curve would be a straight line.  Since $S$ satisfies Conditions \ref{CondSnice}, we know that $S$ is almost quadratic: the Hessian of $S$ obeys $1/2 \le \partial^2 h \le 2$, and the third derivative obeys $| \partial^3 h | \le 10^{-9}$ pointwise.  Therefore, this curve is almost a straight line.  After rotating in the $\omega_1, \omega_2$ plane, it can be given as a graph $\omega_2 = g(\omega_1)$, where $| \nabla g|, | \nabla^2 g|$ are at most $10^{-6}$.  

Next we write $j(\omega_1) = h( \omega_1, g(\omega_1))$.  Because $\partial_1^2 h \ge 1/2$ and $|\nabla g|, |\nabla^2 g|$ are small, it is straightforward to check with the chain rule that 

$$ \partial^2 j \ge 1/4. $$

Let $\omega_1(T_i)$ be the $\omega_1$-coordinate of $\omega(T_i)$.  Equation \ref{T_1+T_2=bar} is equivalent to the following:

\begin{equation} \label{sumsequal}
\omega_1(T_1) + \omega_1(T_2) =  \omega_1(\bar T_1) + \omega_1(\bar T_2) + O(R^{-(1/2) + \delta}). 
\end{equation}

\begin{equation} \label{jsumsequal}
j(\omega_1(T_1)) + j(\omega_1(T_2)) =  j(\omega_1(\bar T_1)) + j(\omega_1(\bar T_2)) + O(R^{-(1/2) + \delta}). 
\end{equation}

Equation \ref{sumsequal} implies that the $\omega_1(T_i)$ and the $\omega_1(\bar T_i)$ have essentially the same midpoint.  Without loss of generality, we can assume that 
$\omega_1(\bar T_1) < \omega_1(T_1) < \omega_1(T_2) < \omega_1(\bar T_2)$.   
Also, since $\omega(T_1)$ lies in (or very near) $\tau_1$, and $\omega(T_2)$ lies in or very near $\tau_2$, $|\omega_1(T_1) - \omega_1(T_2)| \gtrsim K^{-1}$.  Let $I_1$ be the interval $[\omega_1(\bar T_1), \omega_1(T_1)]$ and $I_2$ be the interval $[\omega_1(T_2), \omega_1(\bar T_2)]$.  By Equation \ref{sumsequal}, the lengths of $I_1$ and $I_2$ are equal up to an error of $O(R^{-(1/2)+\delta})$.  
Because of the bound $j'' \ge 1/4$, we see that for any $s_1 \in I_1$ and $s_2 \in I_2$, $j'(s_2) - j'(s_1) \ge (1/4) K^{-1}$.  Using this bound and the fundamental theorem of calculus, we estimate that

$$ |I_1| + |I_2| \lesssim (\int_{I_2} j') - (\int_{I_1} j') + O(R^{-(1/2) + \delta}) = $$

$$ = j(\omega_1(\bar T_2)) - j(\omega_1(T_2)) - j(\omega_1(T_1)) +  j(\omega_1(\bar T_1)) + O(R^{-(1/2) + \delta}) = O(R^{-(1/2) + \delta}). $$

This finishes the proof that $|\omega(T_i) - \omega(\bar T_i)| \lesssim R^{-(1/2)+\delta}$ for $i = 1,2$.  

Next we observe that for each $\theta$, there are only $O(1)$ tubes of $\TT(\theta)$ that intersect $Q$, and so there are only $O(1)$ tubes of $\TT(\theta)$ in $\TT_{j, tang, Q}$.  Therefore, line \ref{cordobasum} is bounded by

\begin{equation} \label{cordobasum2}
 R^{O(\delta)} \sum_{T_1, T_2 \in \TT_{j, tang, Q}} \int | f_{\tau_1, T_1} \dvol_S * f_{\tau_2, T_2} \dvol_S |^2. 
 \end{equation}

Since $\theta(T_1)$ lies in $\tau_1$ and $\theta(T_2)$ lies in $\tau_2$, the angle between the tangent space of $S$ on $\theta(T_1)$ and on $\theta(T_2)$ is $\gtrsim K^{-1}$.  We claim that this angle bound leads to the following inequality: 

\begin{equation} \label{convolbound}
 \int_{\RR^3} | f_{\tau_1, T_1} \dvol_S * f_{\tau_2, T_2} \dvol_S |^2 \lesssim R^{-1/2} \| f_{\tau_1, T_1}  \|_2^2 \|f_{\tau_2, T_2} \|_2^2. 
 \end{equation}

We sketch the proof of the claim.  Let us abbreviate $f_{\tau_1, T_1} \dvol_S$ by $f_1 \dvol_{S_1}$ and $f_{\tau_2, T_2} \dvol_S$ by $f_2 \dvol_{S_2}$, where $S_i$ is a cap containing $\supp f_i$ with radius $\sim R^{-1/2}$.  Because of the angle condition between $S_1$ and $S_2$, we can foliate $S_1$ by curves $\gamma_s$, $s \in [0, R^{-1/2}]$ so that the tangent direction of $\gamma_s$ is quantitatively transverse to the tangent plane of $S_2$, and so that $\dvol_{S_1} = J \cdot \dvol_{\gamma_s} ds$ for a Jacobian factor $J \sim 1$.

We can expand our original function $f_1 \dvol_{S_1} * f_2 \dvol_{S_2}$ as an integral:

$$ f_1 \dvol_{S_1} * f_2 \dvol_{S_2} = \int_0^{R^{-1/2}} (J f_1 \dvol_{\gamma_s} * f_2 \dvol_{S_2} ) ds. $$

Now by Minkowski's inequality and Cauchy-Schwarz,

\begin{equation} \label{minkow}
 \| f_1 \dvol_{S_1} * f_2 \dvol_{S_2}  \|_2^2 \le \left( \int_0^{R^{-1/2}} \| J f_1 \dvol_{\gamma_s} * f_2 \dvol_{S_2} \|_2 ds \right)^2 \le R^{-1/2} \int_0^{R^{-1/2}} \| J f_1 \dvol_{\gamma_s} * f_2 \dvol_{S_2} \|_2^2 ds. 
 \end{equation}

By a change of coordinates argument,

\begin{equation} \label{changecoord}
 \int_{\RR^3} |J f_1 \dvol_{\gamma_s} * f_2 \dvol_{S_2} |^2 \sim \int_{\gamma_s} |f_1|^2 \int_{S_2} |f_2|^2. 
 \end{equation}

Plugging Equation \ref{changecoord} into Equation \ref{minkow}, we get

\begin{equation}  \| f_1 \dvol_{S_1} * f_2 \dvol_{S_2}  \|_2^2 \le
R^{-1/2} \int_0^{R^{-1/2}} (\int_{\gamma_s} |f_1|^2) ds \int_{S_2} |f_2|^2 \lesssim
R^{-1/2} \int_{S_1} |f_1|^2 \int_{S_2} |f_2|^2. 
\end{equation}

This finishes the proof of Equation \ref{convolbound}.  Now using Equation \ref{convolbound} to bound line \ref{cordobasum2}, we see that

$$ \int_{Q} |Ef_{\tau_1, j, tang}|^2 |Ef_{\tau_2, j, tang}|^2 \lesssim R^{O(\delta)} R^{-1/2} (\sum_{T_1 \in \TT_{j, tang, Q}} \| f_{\tau_1, T_1} \|_2^2 ) (\sum_{T_2 \in \TT_{j, tang, Q}} \| f_{\tau_2, T_2} \|_2^2 ) + \neglig . $$

\end{proof}

Next we give an interpretation of Lemma \ref{locbilinear}.  We would like to think of $|Ef_{\tau, T}|$ as well approximated by $\chi_T \| f_{\tau, T} \|_1 \lesssim \chi_T R^{-1/2} \| f_{\tau, T} \|_2$.  Let $S_{\tau, j, tang}$ be a corresponding square function defined as follows:

$$ S_{\tau, j, tang} := \left( \sum_{T \in \TT_{j, tang}} (\chi_T R^{-1/2} \| f_{\tau, T} \|_2)^2 \right)^{1/2}. $$

Lemma \ref{locbilinear} immediately implies that our integral over $Q$ is controlled by the integral with the corresponding square functions:

\begin{equation}
\int_{Q} |Ef_{\tau_1, j, tang}|^2 |Ef_{\tau_2, j, tang}|^2 \lesssim R^{O(\delta)} \int_Q S_{\tau_1, j, tang}^2 S_{\tau_2, j, tang}^2+ \neglig
\end{equation}

Summing over all $Q \subset B_j \cap W$, we get the following bound:

$$ \int_{B_j \cap W} |Ef_{\tau_1, j, tang}|^2 |Ef_{\tau_2, j, tang}|^2 \lesssim R^{O(\delta)} \int_{B_j \cap W} S_{\tau_1, j, tang}^2 S_{\tau_2, j, tang}^2+ \neglig. $$

The last integral involving square functions is easy to bound.  Expanding the definition of square function, we get:

$$\le \sum_{T_1, T_2 \in \TT_{j, tang}} R^{-2} \| f_{\tau_1, T_1} \|_2^2 \|f_{\tau_2, T_2} \|_2^2 \int \chi_{T_1} \chi_{T_2}. $$

Since $T_1$ comes from $\tau_1$ and $T_2$ comes from $\tau_2$, the angle between $v(T_1)$ and $v(T_2)$ is $\gtrsim K^{-1}$, and so the last integral is $\lesssim K R^{3/2}$.  Therefore, the last sum is

$$ \lesssim R^{-1/2} (\sum_{T_1 \in \TT_{j, tang}} \| f_{\tau_1, T_1} \|_2^2) (\sum_{T_2 \in \TT_{j, tang}}  \| f_{\tau_2, T_2} \|_2^2). $$

Using Proposition \ref{wavepack}, the functions $\{ f_{\tau, T} \}_{T \in \TT}$ are almost orthogonal, and we see

$$ \sum_{T \in \TT_{j, tang}} \| f_{\tau, T} \|_2^2 \lesssim \| f_{\tau, j, tang} \|_2^2 + \neglig. $$

Altogether, we have the bound:

\begin{equation*}
\int_{B_j \cap W} |Ef_{\tau_1, j, tang}|^2 |Ef_{\tau_2, j, tang}|^2 \lesssim R^{O(\delta)} R^{-1/2} \| f_{\tau_1, j, tang} \|_2^2 \| f_{\tau_2, j, tang} \|_2^2 + \neglig.
\end{equation*}

This implies the following $L^4$-bound on the bilinear term:

\begin{equation}
\| \BilEft \|_{L^4(B_j \cap W)} \lesssim R^{O(\delta)} R^{-1/8} (\sum_\tau \| f_{\tau, j, tang} \|^2_2)^{1/2} + \neglig. 
\end{equation}

On the other hand we can easily get an $L^2$ bound and then interpolate to get bounds for the $L^p$ norm with any $2 \le p \le 4$.  A standard estimate says that

$$ \| Ef \|_{L^2(B_R)} \lesssim R^{1/2} \| f \|_2. $$

(See for instance Lemma 2.1 in Lecture Notes 7 in \cite{Tnotes}.) 

From this it easily follows that

\begin{equation}
\| \BilEft \|_{L^2(B_j \cap W)} \lesssim R^{1/2} (\sum_\tau \| f_{\tau, j, tang} \|^2_2)^{1/2}
\end{equation}

Interpolating between these by using Holder, we get for all $2 \le p \le 4$, 

\begin{equation} \label{bilinbounda}
\int_{B_j \cap W} | \BilEft|^p \lesssim R^{O(\delta)} R^{\frac{5}{2} - \frac{3}{4}p} (\sum_\tau \| f_{\tau, j, tang} \|_2^2)^{p/2}.
\end{equation}

Next we consider $\| f_{\tau, j, tang} \|_2$.  On the one hand, by Lemma \ref{wavepack2}, we know that $\| f_{\tau, j, tang} \|_2 \lesssim \| f_\tau \|_2$.  We can get a different bound by taking advantage of the small number of directions of tubes in $\TT_{j, tang}$.  
Lemma \ref{tangbound} tells us that $\TT_{j, tang}$ contains tubes in only $R^{O(\delta)} R^{1/2}$ different directions.  Therefore, each function $f_{\tau, j, tang}$ is supported on $R^{O(\delta)} R^{1/2}$ caps $\theta$.  On each cap, Lemma \ref{wavepack2} gives the bound

$$\oint_\theta | f_{\tau, j, tang} |^2 \lesssim \oint_{10 \theta} |f _\tau|^2 \lesssim 1. $$

\noindent Adding the contribution of $R^{(1/2) + O(\delta)}$ caps, we get the bound $\int |f_{\tau, j, tang}|^2 \lesssim R^{O(\delta)} R^{-1/2}$.  Combining these two bounds for $\| f_{\tau, j, tang} \|_2$, we get for $p \ge 3$:

$$ (\sum_\tau \| f_{\tau, j, tang} \|_2^2)^{p/2} \le R^{O(\delta)} R^{\frac{3}{4} - \frac{p}{4}} (\sum_\tau \| f_{\tau, j, tang} \|_2^2)^{3/2}. $$

Substituting this bound into Equation \ref{bilinbounda}, we get:

$$ \int_{B_j \cap W} | \BilEft|^p \lesssim R^{O(\delta)} R^{\frac{13}{4} - p} (\sum_\tau \| f_{\tau} \|_2^2)^{3/2}. $$

Taking $p = 3.25 = 13/4$, this estimate is the bound in Proposition \ref{tangtermbound}.

\section{Estimates about the geometry of tubes and algebraic surfaces}

In this section, we prove Lemmas \ref{transbound} and \ref{tangbound}.   These Lemmas estimate how tubes interact with an algebraic surface.  Each Lemma generalizes a simple statement about lines intersecting an algebraic surface.

A line can transversally intersect a degree $D$ surface $Z(P)$ in at most $D$ points.  Lemma \ref{transbound} says that a tube $T$ can belong to $\TT_{j, trans}$ for at most $\Poly(D)$ values of $j$: there are $\le \Poly(D)$ balls $B_j$ where $T$ passes through $W$ transversally.

The directions of the lines in an algebraic surface $Z(P)$ all lie in an algebraic curve.  Let $\mathbb{RP}^2$  denote the points at infinity in $\RR^3$ -- also the set of directions of lines in $\RR^3$.  The projective closure of $Z(P)$ intersects $\mathbb{RP}^2$ in an algebraic curve.  If a line $l$ lies in $Z(P)$, then the direction of the line must lie in this curve in $\mathbb{RP}^2$.  Lemma \ref{tangbound} says that the tubes of $\TT_{j, tang}$ contain tubes from at most roughly $R^{1/2}$ of the $R$ caps $\theta$.  If $S$ is a sphere, this is roughly the number of caps that would intersect an algebraic curve of degree $D$ in $S$.

Transferring ideas from lines to tubes is sometimes straightforward and sometimes hard.  Some of the methods that we use here come from the paper \cite{Gu2}.

\subsection{Bounding transversal intersections}

We begin with the estimate for transversal tubes, Lemma \ref{transbound}.  Suppose that $T \in \TT$.  Recall from Definition \ref{deftrans} that if $T \in \TT_{j, trans}$, then there is a non-singular point $z \in 10 T \cap 2 B_j \cap Z(P)$ so that $\Angle( v(T), T_z Z) > R^{-(1/2) + 2 \delta}$.  We have to prove that any tube $T \in \TT$ lies in $\TT_{j, tang}$ for $\le \Poly(D)$ values of $j$.  We state a slightly more general result.  

\begin{lemma}  \label{geomtransbound} Suppose that 

\begin{itemize}

\item $T$ is a finite cylinder in $\RR^3$ with radius $\rho$ and arbitrary length.  

\item $a \in (0, 1/10)$ denotes an angle.  

\item $T$ is subdivided into tube segments of length $\ge \rho a^{-1}$.

\item $Q$ is a non-singular polynomial of degree $D$.

\item $Z_{\ge a}(Q) := \{ z \in Z(Q) | \Angle(v(T), T_z Z(Q)) \ge a \}. $

\end{itemize}

Then $Z_{\ge a}(Q) \cap T$ is contained in $\lesssim D^3$ of the tube segments of $T$.

\end{lemma}

The reader may want to imagine $\rho =1$ and $a = 1/10$.  (The general case can be reduced to this case by a change of coordinates.  On the other hand, it is just as easy to prove the lemma for all $\rho$ and $a$ as stated, so we give the proof in the general case.)

To see that this Lemma implies Lemma \ref{transbound}, we first note that $P$ is a product of non-singular irreducible polynomials.  For each of these polynomials, we apply the Lemma above to $10T$, taking $\rho = 10 R^{(1/2)+\delta}$, $a = R^{-(1/2)+2 \delta}$, and the length of the segments $\rho a^{-1} = 10 R^{1 - \delta}$.  So each segment intersects $O(1)$ balls $B_j$.  (This step motivates the choice of angle $R^{-(1/2) + 2 \delta}$ in the definitions of $\TT_{j, tang}$ and $\TT_{j, trans}$.)

There is probably a version of this lemma in any number of dimensions, but we will focus on 3 dimensions.  In fact, we'll warm up by proving a 2-dimensional version of the lemma, and then go on to the more difficult 3-dimensional case.  We begin with a lemma that holds in any number of dimensions.

If $T$ is a tube in $\RR^n$ in direction $v(T)$, and $Q$ is a non-singular polynomial on $\RR^n$, then we define $Z_{=a}(Q)$ as follows:

$$ Z_{=a}(Q) := \{ z \in Z(Q) | \Angle(v(T), T_z Z(Q)) = a \}. $$

We defined earlier a non-singular polynomial.  Recall that we said that a polynomial $P$ on $\RR^n$ is non-singular if for each point $x \in Z(P)$, $\nabla P(x) \not= 0$.  There is an analogous definition for varieties defined by several polynoimals.  Suppose that $Q_1, ..., Q_k$ are polynomials on $\RR^n$.  We say that $Z(Q_1, ..., Q_k)$ is a transverse complete intersection if for each point $x \in Z(Q_1, ..., Q_k)$, $\nabla Q_1(x), ..., \nabla Q_k(x)$ are linearly independent.  In particular, a transverse complete intersection is always a smooth submanifold of dimension $n-k$.  

\begin{lemma} \label{Z_=a} Suppose $Q$ is a non-singular polynomial on $\RR^n$.  For any $a$, $Z_{=a}(Q)$ is a variety $Z(Q, Q_1)$
where $Q_1$ is a polynomial (depending on $Q$ and $a$) of degree $\lesssim \Deg(Q)$.  For almost every $a$, $Z(Q, Q_1)$ is a transverse complete intersection.  \end{lemma}

\begin{proof}
Suppose $x \in Z(Q)$.  Since $Q$ is non-singular, $\nabla Q(x) \not= 0$.  The unit normal to $Z(Q)$ at $x$ is given by $\pm \frac{\nabla Q} {|\nabla Q|}$.  Therefore, $x \in Z_{=a}(Q)$ if and only if

$$ \frac{\nabla Q} {|\nabla Q|} \cdot v(T) = \pm \sin a. $$

This holds if and only if

$$0 =(\nabla Q \cdot v(T))^2 -  \sin^2 (a) |\nabla Q|^2  =: Q_1. $$

We see that $Q_1$ is a polynomial and that $Z_{=a}(Q) = Z(Q, Q_1)$.  

Next we want to see that for almost every $a$, for each point $x \in Z_{=a}(Q)$, $\nabla Q$ and $\nabla Q_1$ are linearly independent.

Define a smooth function $f: Z(Q) \rightarrow \RR$ by

$$ f = \frac{(\nabla Q \cdot v(T) )^2}{|\nabla Q|^2}. $$

We note that $|\nabla Q|$ never vanishes on $Z(Q)$, so $f$ is $C^\infty$ smooth.  Also $f(x) = \sin^2(a)$ if and only if $x \in Z_{=a}(Q)$.  

Fix any value of $a$.  If $x_0 \in Z_{=a}(Q)$, and $Q_1$ is defined as above, then we claim that $\nabla Q$ and $\nabla Q_1$ are linearly dependent at $x_0$ if and only if $\nabla f(x_0) = 0$.  We can see this as follows.  Along the manifold $Z(Q)$, the polynomial  $Q_1$ is equal to

$$Q_1 (x) = |\nabla Q|^2 ( f(x) - \sin^2(a) ). $$

\noindent At the point $x_0$, $f(x_0) - \sin^2 (a) = 0$.  So when we differentiate, we see that 

$$ \nabla Q_1(x_0) = |\nabla Q(x_0)|^2 \nabla f(x_0). $$

\noindent We have $\nabla Q(x_0), \nabla Q_1(x_0)$ linearly independent as vectors in $\RR^n$ if and only the restriction of $\nabla Q_1(x_0)$ to $T_{x_0} Z(Q)$ is non-zero, if and only if $\nabla f(x_0) \not= 0$.

Now by Sard's theorem, the set of critical values of $f$ has measure zero.  Therefore, for almost every $a$, $\sin^2(a)$ is a regular value of $f$.  For any such $a$, $\nabla Q$ and $\nabla Q_1$ are linearly independent at every point of $Z_{=a}(Q) = Z(Q, Q_1)$.  
\end{proof}

In this section we will also use Bezout's theorem.  We use the following version -- see Theorem 5.2 in \cite{bezout?} for a clean and well-written proof.

\begin{theorem} \label{bezout} If $Z(Q_1, ..., Q_n)$ is a transverse complete intersection in $\RR^n$, then the number of points in $Z(Q_1, ..., Q_n)$ is at most $\Deg(Q_1) ... \Deg(Q_n)$.  
\end{theorem}

Now we can prove a 2-dimensional version of Lemma \ref{geomtransbound}.

\begin{lemma}  \label{geomtransbound2d} Suppose that 

\begin{itemize}

\item $T$ is a rectangle in $\RR^2$ with width $2 \rho$ and arbitrary length.  

\item $a \in (0, 1/10)$ denotes an angle.  

\item $T$ is subdivided into rectangular segments of length $\ge \rho a^{-1}$.

\item $Q$ is a non-singular polynomial of degree $D$.

\item $Z_{\ge a}(Q) := \{ z \in Z(Q) | \Angle(v(T), T_z Z(Q)) \ge a \}. $

\end{itemize}

Then $Z_{\ge a}(Q) \cap T$ is contained in $\lesssim D^2$ of the tube segments of $T$.

\end{lemma}

\begin{proof} Using Lemma \ref{Z_=a}, we choose a generic $b \in [(9/10) a, a]$ so that $Z_{=b}$ is a transverse complete intersection.  Since we are working in 2 dimensions, $Z_{=b}$ is a set of $\lesssim D^2$ points.

We choose coordinates $x_1, x_2$ so that $T$ is defined by $|x_2| \le \rho$.  The $x_1$-axis is parallel to the long side of $T$, so $v(T) = \partial_1$.  We say that a point $x \in Z(Q)$ is vertical if $T_x Z(Q)$ is parallel to the $x_2$-axis, or equivalently if $\partial_2 Q = 0$.  By making a tiny perturbation of $T$, we can assume that $Z(Q, \partial_2 Q)$ is also a transverse complete intersection, and so consists of $\le D^2$ points.

We divide $2T$ into tube segments corresponding to the original tube segments.  We label a tube segment bad if it lies within $10 \rho a^{-1}$ of a vertical point or a point of $Z_{=b}$.  The total number of bad segments is $\lesssim D^2$.

Suppose that $x \in Z_{\ge a}(Q) \cap T$ and that $x$ is not in any of the bad tube segments.  We consider the connected component of $Z(Q) \cap (2 T \setminus \textrm{ bad segments})$ that contains $x$ -- call this component $Z_{comp}$.  The curve $Z_{comp}$ contains no vertical points.  Therefore, it is defined as a graph $x_2 = h(x_1)$ for a smooth function $h: I \rightarrow \RR$ on some interval $I$.  Also, $Z_{comp}$ does not contain any points of $Z_{=b}(Q)$.  Since $x \in Z_{\ge a}(Q) \subset Z_{\ge b}(Q)$, we see that $Z_{comp} \subset Z_{> b} (Q)$.  Therefore, $|\nabla h| \ge \sin b \ge (1/2) a$ at every point of the interval $I$.  Since $\nabla h$ is continuous, its sign must be constant.  Therefore the length of $I$ is $\le 10 \rho a^{-1}$, and $Z_{comp}$ can be covered by $\lesssim 1$ tube segments.  

It remains to prove that all these components $Z_{comp}$ can be covered by $\lesssim D^2$ tube segments.  Some of the components $Z_{comp}$ have $\partial Z_{comp}$ that intersects the boundary of a bad tube segment.  Since there are $\lesssim D^2$ bad tube segments, all such components can be covered by $\lesssim D^2$ tube segments.  For other components $\partial Z_{comp}$ does not intersect the boundary of a bad tube segment.  In this case, the two boundary points of $Z_{comp}$ must lie on the top and bottom of the rectangle $T$.  In this case, $Z_{comp}$ ``goes across'' the rectangle $T$.  For $|h| < \rho$, the line $x_2 = h$ must intersect $Z_{comp}$, and for almost every $h$, it must intersect $Z_{comp}$ transversally.  Since any line has at most $D$ transverse intersections with $Z(Q)$, the number of such components is at most $D$. \end{proof}

Our 3-dimensional result, Lemma \ref{geomtransbound}, is more complicated than this 2-dimensional model.  In 2 dimensions, $Z_{=b}(Q)$ was a set of points of controlled cardinality.  But in 3 dimensions, $Z_{=b}(Q)$ will be a curve.  The next step in approaching our 3-dimensional Lemma is to prove a result about algebraic curves in a 3-dimensional tube.  We will use this result to control the curve $Z_{=b}(Q)$ (and some other curves).

If $Y = Z(Q_1, Q_2) \subset \RR^3$ is a transverse complete intersection, then we define

$$ Y_{\ge a} := \{ y \in Y | \Angle (v(T), T_y Y) \ge a \}. $$

\begin{lemma}  \label{geomtransboundcurve3d} Suppose that 

\begin{itemize}

\item $T$ is a finite cylinder in $\RR^3$ with radius $\rho$ and arbitrary length.  

\item $a \in (0, 1/10)$ denotes an angle.  

\item $T$ is subdivided into tube segments of length $\ge \rho a^{-1}$.

\item $Y = Z(Q_1, Q_2)$ is a transverse complete intersection.

\item $Q_1$ and $Q_2$ have degree at most $D$.

\end{itemize}

Then $Y_{\ge a} \cap T$ is contained in $\lesssim D^3$ of the tube segments of $T$.

\end{lemma}

We start by studying $Y_{=a}$ and proving a version of Lemma \ref{Z_=a} for cureves in $\RR^3$.

\begin{lemma} \label{Y_=a} Suppose that $Y = Z(Q_1, Q_2)$ is a transverse complete intersection in $\RR^3$ and that $Q_1$ and $Q_2$ have degree at most $D$.  Then $Y_{=a}$ is an algebraic variety of the form $Z(Q_1, Q_2, Q_a)$, where $Q_a$ is a polynomial (depending on $Q_1, Q_2$, and $a$) of degree $\lesssim D$.  Moreover, for almost every $a$, $Z(Q_1, Q_2, Q_a)$ is a transverse complete intersection.  In particular,$Y_{=a}$ consists of $\lesssim D^3$ points. 
\end{lemma}

\begin{proof} If $y \in Y = Z(Q_1, Q_2)$, then the vector $\nabla Q_1(x) \times \nabla Q_2(x)$ spans $T_y Y$.  Therefore, we have $\Angle ( v(T), T_y Y) = a$ if and only if

$$ 0= \left( (\nabla Q_1 \times \nabla Q_2) \cdot v(T) \right)^2 - \cos^2 (a) | \nabla Q_1 \times \nabla Q_2|^2 =: Q_a. $$

This proves the first claim.  Now we argue as in the proof of Lemma \ref{Z_=a}.  We define
a function $f: Y \rightarrow \RR$ by

$$ f = \frac{\left( (\nabla Q_1 \times \nabla Q_2) \cdot v(T) \right)^2 }{ | \nabla Q_1 \times \nabla Q_2|^2 }, $$

\noindent so that $f(y) = \cos^2 (a)$ if and only if $y \in Y_{=a}$.  Fix $a$ and suppose that $y_0 \in Y_{=a}$.  We can write $Q_a$ as

$$ Q_a(y) =  | \nabla Q_1 \times \nabla Q_2|^2 \left( f(y) - \cos^2 (a) \right). $$

\noindent Since $f(y_0) - \cos^2 (a) = 0$, we see that $\nabla Q_1, \nabla Q_2, \nabla Q_a$ are linearly independent at $y_0$ if and only if $\nabla f(y_0) \not= 0$, where $\nabla f$ is considered as a vector field on $Y$.  By Sard's theorem, the critical values of $f$ have measure 0.  For almost every $a$, $\cos^2(a)$ is a regular value of $f$, and so $Z(Q_1, Q_2, Q_a)$ is a transverse complete intersection.

If $Z(Q_1, Q_2, Q_a)$ is a transverse complete intersection, then Bezout's theorem implies that it consist of $\lesssim D^3$ points. 
\end{proof}

Now we can begin the proof of Lemma \ref{geomtransboundcurve3d}.

\begin{proof} By Lemma \ref{Y_=a}, we can choose $b \in [(9/10) a, a]$ so that $Y_{=b}$ consists of $\lesssim D^3$ points.

Choose coordinates $x_1, x_2, x_3$ so that $T$ is given by the equation $(x_2, x_3) \in B^2(0, \rho)$.  In these coordinates $v(T) = \partial_1$.  Define $Y_{e_i^\perp}$ to be the set of points $y \in Y$ where $T_y Y \subset e_i^\perp$.  $Y_{e_i^\perp}$ is a variety: it is equal to $Z(Q_1, Q_2, (\nabla Q_1 \times \nabla Q_2) \cdot e_i )$.  After a small generic rotation of $T$ (and hence the coordinates), we can assume that it is a transverse complete intersection and so it consists of $\lesssim D^3$ points.

We divide $2T$ into tube segments corresponding to the original tube segments.  We label a tube segment bad if it lies within $10 \rho a^{-1}$ of a point of $Y_{=b}$ or $Y_{e_i^\perp}$.  The total number of bad segments is $\lesssim D^3$.

Suppose that $y \in Y_{\ge a} \cap T$ and that $y$ is not in any of the bad tube segments.  We consider the connected component of $Y \cap (2 T \setminus \textrm{ bad segments})$ that contains $y$ -- call this component $Y_{comp}$.  The curve $Y_{comp}$ contains no points of $Y_{e_1^\perp}$.  Therefore, it is defined as a graph $(x_2, x_3) = (h_2(x_1), h_3(x_1))$ for a smooth function $h = (h_2, h_3): I \rightarrow \RR^2$ on some interval $I$.  Since $Y_{comp}$ contains no points of $Y_{e_2^\perp}$ or $Y_{e_3^\perp}$, the sign of $\frac{dh_2}{dx_1}$ is constant and the sign of $\frac{dh_2}{dx_1}$ is constant.
Also, $Y_{comp}$ does not contain any points of $Y_{=b}$.  Since $y \in Y_{\ge a}$, we see that $Y_{comp} \subset Y_{> b}$.  Therefore, at every point of the interval $I$, 

\begin{equation} \label{curveangle}
\left| \frac{dh_2}{dx_1} \right| + \left| \frac{dh_3}{dx_1} \right| \ge (1/10) a.
\end{equation}

\noindent Therefore the length of $I$ is $\le 100 \rho a^{-1}$, and $Y_{comp}$ can be covered by $\lesssim 1$ tube segments.  

We have to prove that the set of such $Y_{comp}$ can be covered by $\lesssim D^3$ tube segments.  Some of the $Y_{comp}$ have a boundary point in the boundary of a bad tube segment.  The set of all such $Y_{comp}$ can be covered by $\lesssim D^3$ tube segments.

We consider components $Y_{comp}$ with no boundary point in a bad segment.  Recall that $Y_{comp}$ contains a point of $T$, and the boundary of $Y_{comp}$ must lie in $\partial (2T)$.  Let $I = (s_1, s_2)$.  Then either $| h_2(s_1) - h_2(s_2) | \ge \rho$ (type 2) or $|h_3(s_1) - h_3(s_2)| \ge \rho$ (type 3).

Each component of type 2 intersects many planes of the form $x_2 = h$.  For each type 2 component, the plane $x_2 = h$ intersects $Y_{comp}$ transversely for $h$ in a subinterval of $[- 2 \rho, 2 \rho]$ of measure at least $\rho$.  By Bezout's theorem, there are at most $D^2$ points where $Y$ intersects a plane transversely, and so the total number of type 2 components is at most $4 D^2$.  The number of type 3 components is also at most $4 D^2$.  
\end{proof}

Now we can begin the proof of the main result of this subsection, Lemma \ref{geomtransbound}.

\begin{proof} By Lemma \ref{Z_=a}, we can choose an angle $b \in [(9/10)a, a]$ so that $Z_{=b}$ is a transverse complete intersection of polynomials of degree $\lesssim D$.  

We remark that if $x \in Z_{=b}$, then $\Angle( v(T), T_x Z_{=b}) \ge b$.  We state this as a general observation.  Suppose that $Y \subset Z$ is a smooth curve and $x \in Y$.  Recall that the angle $\Angle (v(T), T_x Z)$ is defined to be $\min_{0 \not= w \in T_x Z} \Angle( v(T), w)$.  Since $T_x Y \subset T_x Z$, we get

\begin{equation} \label{anglecomparison}
\Angle (v(T), T_x Y) \ge \Angle( v(T), T_x Z).
\end{equation}

\noindent  In particular, if $x \in Z_{=b}$, we see that $\Angle( v(T), T_x Z_{=b}) \ge \Angle( v(T), T_x Z) = b$.  So if $Y = Z_{=b}$, then $Y_{\ge b}$ is all of $Y$.  

Now by Lemma \ref{geomtransboundcurve3d}, $Z_{=b} \cap 10T$ can be covered by $\lesssim D^3$ tube segments.  

Next we consider some other curves in $Z$.  For any non-zero vector $w$, we define
$ \\Tan_w \subset Z$ by

$$ \\Tan_w := \{ x \in Z | w \in T_x Z \} = Z(Q, \nabla Q \cdot w). $$

For almost every $w$, $\Tan_w = Z(Q, \nabla Q \cdot w)$ is a transverse complete intersection.  We let $W$ be a set of $O(1)$ unit vectors, including the coordinate vectors $e_1, e_2, e_3$, forming a $1/1000$-net on $S^2$.  We will say more about the choice of $W$ below.  After a tiny rotation of coordinates, we can assume that $\Tan_w$ is a transverse complete intersection for every $w \in W$.  By Lemma \ref{geomtransboundcurve3d}, the $|W|$ curves $(\Tan_w)_{\ge b} \cap 10T$ can be covered by $\lesssim D^3$ tube segments.

We divide $10T$ into tube segments corresponding to the original tube segments.  We label a tube segment bad if it lies within $100 \rho a^{-1}$ of a point of $Z_{=b}$ or $(\Tan_w)_{\ge b}$ for some $w \in W$.  The total number of bad segments is $\lesssim D^3$.

Suppose that $x \in Z_{\ge a} \cap T$ and that $x$ is not in any of the bad tube segments.  We consider the connected component of $Z \cap 2 T \cap B(x, 20 \rho a^{-1})$
that contains $x$ -- call this component $Z_{comp}$.  We know that $Z_{comp}$ contains no point of $Z_{=b}$, and so $Z_{comp} \subset Z_{> b}$.  

We also know that $Z_{comp}$ contains no point of $(\Tan_w)_{\ge b}$.  We claim that $Z_{comp}$ contains no point of $\Tan_w$.  Suppose that $x \in Z_{comp} \cap \Tan_w$.  Since $x \in Z_{comp}$, we have just seen that $\Angle(v(T), T_x Z) > b$.  But by equation \ref{anglecomparison}, we know that 

$$\Angle ( v(T), T_x (\Tan_w)) \ge \Angle (v(T), T_x Z) > b. $$  

\noindent Therefore, we would have $x \in (\Tan_w)_{\ge b}$.  So we conclude that $Z_{comp}$ contains no point of $\Tan_w$.

Since $W$ includes a $(1/1000)$-net of unit vectors, and $Z_{comp}$ does not intersect $\cup_{w \in W} \Tan_w$, it follows that the tangent plane $T_z Z$ is almost constant as $z$ varies in $Z_{comp}$: the tangent plane can only vary by an angle at most $1/100$.  

To finish the proof of Lemma \ref{geomtransbound}, we have to prove the following intersection estimate for $Z_{comp}$.  Consider lines parallel to the $x_1$-axis of the form $x_2 = h_2, x_3 =h_3$ with $(h_2, h_3) \in B^2(2 \rho)$.  We want to prove that for a subset of $B^2(2 \rho)$ with area $\ge \rho^2$, the corresponding line intersects $Z_{comp}$.  

Suppose for a moment that we have such an intersection estimate.  We claim that there are at most $4 \pi D$ points of $Z_{\ge a} \cap T$ that lie outside of the bad segments and are pairwise separated by $100 \rho a^{-1}$.  To prove the claim, suppose that we had more than $4 \pi D$ such points.  Consider the surface $Z_{comp}$ around each of the points -- because the points are separated, these surfaces are disjoint patches of $Z$.  By an averaging argument, we can find $(h_2, h_3) \in B^2(2 \rho)$ so that the line $x_2 = h_2$, $x_3 = h_3$ intersects more than $D$ of the surfaces $Z_{comp}$.  Also, the set of $(h_2, h_3)$ so that the line $x_2 = h_2$, $x_3 = h_3$ intersects $Z$ non-transversally has measure 0, so we can assume that our line intersects $Z$ transversally at more than $D$ points.  This gives a contradiction, proving our claim.

Given this claim, the portion of $Z_{\ge a} \cap T$ outside of the bad segments can be covered by $\lesssim D$ tube segments.  Since there are $\lesssim D^3$ bad tube segments, $Z_{\ge a} \cap T$ can be covered by $\lesssim D^3$ tube segments in total.  So it only remains to prove the intersection estimate.

Recall that the tangent plane of $Z_{comp}$ is nearly constant.  In the main case, $\Angle( v(T), T_z Z) \le (1/10)$ for all $z \in Z_{comp}$.  Let us first handle this case.  Because $Z_{comp}$ does not intersect $Tan_{e_3}$, at each point $z \in Z_{comp}$, the tangent plane $T_z Z$ can be given as a graph of the form $x_3 = L_z (x_1, x_2)$.  Because $Z_{comp} \subset Z_{\ge b}$, we know that $(9/10)a \le \Angle(v(T), T_z Z)$.  Being in the main case, we have also assumed that $\Angle( v(T), T_z Z) \le 1/10$.  Therefore, we get the following inequalities about $L_z$:

$$ a/2 \le |L_z(1,0)| \le 1/10.$$

We would also like to know something about $L_z (0,1)$.  The tangent plane $T_z Z$ is almost constant on $Z_{comp}$, so if there happens to be a single point $z_0 \in Z_{comp}$ where $|L_{z_0}(0,1)| \le 1/2$, then $|L_z(0,1)| \le 1$ for all $z \in Z_{comp}$.  We can arrange this by performing a rotation in the $x_2-x_3$ plane by an angle which is a multiple of $\pi/10$.  These rotations generate a finite group, so we can also assume that $W$ is invariant with respect to any of these rotations.  After the rotation, we still have $(9/10) a \le \Angle(v(T), T_z Z) \le (1/10)$ for all $z \in Z_{comp}$.  Therefore, without loss of generality we can arrange that for every $z \in Z_{comp}$, $L_z$ obeys the bounds

$$ a/2 \le |L_z(1,0)| \le 1/10; |L_z(0,1)| \le 1. $$

Recall that $Z_{comp}$ was defined around an original point $x \in Z_{\ge a} \cap T$.  We let $\pi$ be a plane through $x$, perpendicular to $v(T)$.  We can assume without loss of generality that the $x_1$ coordinate of the original point $x$ is zero, so that the plane $\pi$ is defined by $x_1 = 0$.  The intersection $\pi \cap T$ is a disk of radius $\rho$ centered at $x$, and $\pi \cap T \cap Z_{comp}$ is a smooth curve in this disk.  (Since $Tan_{e_2} \cap Z_{comp}$ is empty, $Z_{comp}$ is transverse to $\pi$.)  We look at the component of this curve containing the point $x$.  Because of the bound $|L_z(0,1)| \le 1$, this component can be given by a graph of the form $x_3 = g(x_2)$, for a function $g$ with $|\nabla g| \le 1$.  The function $g$ is defined on an interval containing $[-\rho/2, \rho/2] =: I_2$.  On $I_2$, we have $|g(x_2)| \le \rho/2$.  

For each $b_2 \in I_2$, consider the intersection of $Z_{comp}$ with the plane $x_2 = b_2$.  Since $Z_{comp}$ is disjoint from $Tan_{e_3}$, the intersection is a smooth curve.  Consider the connected component of this intersection which contains the point $(0, b_2, g(b_2))$.  Since $a/2 \le |L_z(1,0)| \le 1/10$,  this connected component is given by a graph of the form $x_2 = b_2$, $x_3 = j_{b_2}(x_1)$, where $|\nabla j| \ge a/2$.  By continuity the sign of $\frac{dj}{dx_1}$ must be constant.  We also know that $|j(0)| = |g(b_2)| \le \rho/2$, and $|b_2| \le \rho_2$.  The function $j$ is defined on an interval $I_1(b_2)$.  Let $e_1$ be the positive endpoint of $I_1(b_2)$.  Recalling the definition of $Z_{comp}$, we see that either $(b_2, j_{b_2}(e_1)) \in \partial B(2 \rho)$ or else $e_1 \ge 2 \rho a^{-1}$.  In either case, the image of $j_{b_2}$ must cover an interval $I_3(b_2)$ of length $\ge \rho$.  In the first case, we have $|j_{b_2}(e_1)| \ge (3/2) \rho$ and $|j_{b_2}(0)| \le (1/2) \rho$.  In the second case, since $| \nabla j| \ge a/2$ and $j$ is defined on $[0, 2 \rho a^{-1}]$, the image of $j$ must again cover an interval of length $\rho$.  

We have seen that $Z_{comp}$ intersects the line $x_2 =b_2$, $x_3 = b_3$ whenever $b_2 \in I_2$ and $b_3 \in I_3(b_2)$.  The total area of this region is $\ge \rho^2$.  This completes the proof of the intersection estimate in the main case that $\Angle(v(T), T_z Z) \le (1/10)$ for all $z \in Z_{comp}$.

Next we consider the minor case that $\Angle( v(T), T_z Z) \ge (1/20)$ for all $z \in Z_{comp}$.  In this case, $T_z Z'$ is a graph of the form $x_1 = \bar L_z (x_2, x_3)$ where $\bar L_z$ is a linear function obeying 

$$ |\bar L_z (x_2, x_3) | \le 40 | (x_2, x_3)|. $$

In this case, $Z_{comp}$ is a graph of the form $x_1 = h(x_2, x_3)$ over the disk $B^2(2 \rho)$ in the $x_2-x_3$ plane with $| \nabla h| \le 40$.  But in this case, $Z_{comp}$ intersects every line of the form $x_2 = b_2, x_3 = b_3$ with $(b_2, b_3) \in B^2(2 \rho)$.  This finishes the proof of the intersection estimate and hence the proof of Lemma \ref{geomtransbound}.

\end{proof}

\subsection{Directions of tangential tubes}

In this section, we prove Lemma \ref{tangbound}.  The main tool in the proof is a theorem of Wongkew \cite{W} on the volumes of neighborhoods of real algebraic varieties.  Here is a special case of the theorem.

\begin{theorem} \label{Wongkew} (Wongkew) If $P$ is a non-zero polynomial of degree $D$ on $\RR^n$, then

$$\Vol \left( B(L) \cap N_\rho Z(P) \right)  \le C_n D \rho L^{n-1}. $$
\end{theorem}

Remark. Recently, Zhang gave an application of Wongkew's theorem in incidence geometry \cite{Z}.  

We need a minor generalization where the ball $B(L)$ is replaced by a rectangular region.

\begin{theorem} \label{Wongkew'} Suppose that $R$ is an $n$-dimensional rectangular grid of unit cubes with dimensions $R_1 \times ... \times R_n$, where $1 \le R_1 \le ... \le R_n$.  Suppose that $P$ is a non-zero polynomial of degree $D$.  Then the number of cubes of the grid that intersect $Z(P)$ is at most $C_n D \prod_{j=2}^n R_j$.  
\end{theorem}

The proof of this theorem is a minor modification of Wongkew's proof.  

\begin{proof} The proof is by induction on $n$.  When $n=1$, the theorem reduces to the fact that a degree $D$ polynomial in one variable has at most $D$ zeroes.

By a theorem of Oleinik-Petrovskii, Milnor, and Thom \cite{M}, the number of connected components of $Z(P)$ is $\le C_n D^n$.  Therefore, the number of cubes that contain a connected component of $Z(P)$ is $\lesssim D^n$.   If a cube intersects $Z(P)$ and does not contain a component of $Z(P)$, then one of its boundary faces intersects $Z(P)$.  We will count cubes of this type by induction on the dimension.

Now we want to count $(n-1)$-faces of the grid that intersect $Z(P)$.  
Consider all the $(n-1)$-dimensional rectangular grids formed from $R$ by fixing one of the coordinates to an integer value.  For each $j$, there are $R_j + 1$ such rectangular grids formed by intersecting $R$ with planes of the form $x_j = h_j$, $h_j = 0, ..., R_j$. The polynomial $P$ may vanish on at most $D$ of these $(n-1)$-planes, contributing at most $D \prod_{j=2}^n R_j$ $(n-1)$-dimensional faces.  If $P$ does not vanish on one of these $(n-1)$-dimensional rectangular grids, then we can use induction to bound the number of $(n-1)$-faces of this $(n-1)$-dimensional grid that intersect $Z(P)$.

For $j \not=1$, there are $\lesssim R_j$ rectangular grids in the $e_j^\perp$ direction.  In each of these grids, $Z(P)$ may intersect at most $C_{n-1} D R_j^{-1} \prod_{j'=2}^n R_{j'}$ $(n-1)$-faces.  Altogether, this contributes $\lesssim_n D \prod_{j=2}^n R_j$ $(n-1)$-faces.

For $j=1$, the bound is even better.  There are $\lesssim R_1$ rectangular grids in the $e_1^\perp$ direction.  In each of these grids, $Z(P)$ may intersect at most $C_{n-1} D \prod_{j'=3}^n R_{j'}$ $(n-1)$-faces.  So the total number of $(n-1)$-faces of this orientation is $\lesssim_n D R_1 R_3 R_4 ... R_n \le D \prod_{j=2}^n R_j$.  

\end{proof}

Now we set up Lemma \ref{tangbound} in a slightly more general way.

Suppose that $B= B^3(L)$ is a 3-dimensional ball of radius $L$.  Let $P$ be a product of non-singular polynomials of degree at most $D$, and let $Z = Z(P)$.  Let $\TT$ be a set of cylindrical tubes $T$ of thickness $\rho$.  We say that $T \in \TT$ lies in $\TT_{tang}$ if $2 T \cap Z \cap (1.1) B \not= \phi$ and for each non-singular point $x \in 10 T \cap Z \cap 2 B$,

$$ \Angle( v(T), T_x Z) \le \rho / L . $$

We say that two tubes $T_1, T_2 \in \TT$ point in different directions if the angle between $v(T_1)$ and $v(T_2)$ is at least $\rho / L$.  

\begin{lemma} \label{geomtangbound} If $\TT' \subset \TT_{tang}$ are tubes pointing in pairwise different directions, then 

$$ | \TT' | \lesssim D^2 \log^2 (L/ \rho) L / \rho. $$

\end{lemma}

(To recover Lemma \ref{tangbound}, we take, $B = B_j$, $L = R^{1 - \delta}$, and $\rho = R^{(1/2) + \delta}$.  Therefore, if $\TT' \subset \TT_{j, tang}$ consists of tubes that point in $R^{-(1/2) + 2 \delta}$-separated directions, then Lemma \ref{geomtangbound} guarantees that $|\TT'| \lesssim R^{(1/2) + O(\delta)}$.  Next if we let $\TT''$ be a subset of $\TT_{j, tang}$ consisting of tubes that point in $R^{-1/2}$-separated directions, then the bound for $|\TT''|$ is at most $R^{O(\delta)}$ times larger than the bound for $|\TT'|$.  In conclusion, the number of different directions of tubes in $\TT_{j, tang}$ is $\lesssim R^{(1/2) + O(\delta)}$. )

\begin{proof}

By scaling we can assume that $\rho = 1$.  

If $T \in \TT_{tang}$, then we claim that $T \cap (3/2) B$ is contained in the $10$-neighborhood of $Z(P)$.  By assumption, there is a non-singular point $z_0 \in 2T \cap (1.1) B \cap Z(P)$.  Since $P$ is a product of non-singular varieties, $Z(P)$ is a union of smooth varieties $Z(P_l)$, and $z_0$ is in exactly one of them -- say $z_0 \in Z(P_l)$.  

We draw a curve in $Z(P_l)$ starting at $z_0$ and trying to stay as close as possible to the core line of $T$.  We choose coordinates $x_1, x_2, x_3$ where $T$ is given by $x_2^2 + x_3^2 \le 1$ and where the center of $B$ has $x_1$-coordinate 0.  Now at each point $z$ of $Z(P_l) \cap 10 T \cap 2 B$, $\Angle (v (T), T_z Z(P_l)) \le \rho / L = 1 / L$.  Therefore, we can parametrize a curve in $Z(P_l)$ starting at $z_0$, given by a graph $(x_2, x_3) = g(x_1)$ with $|\nabla g| \le 1 / L$, for $|x_1| \le (3/2)L$.  This curve lies in $Z(P)$, and $T \cap (3/2) B$ lies in the $10$-neighborhood of the curve.

The rest of the proof is a hairbrush argument, following Wolff's hairbrush idea from \cite{W2}.  Suppose that $|\TT'| = \beta L $.  We will prove that 
$\beta \lesssim  D^2 \log^2 L$.   For each tube $T$, $\Vol (B \cap T) \sim L$.  We cover $N_{10} Z(P) \cap B$ with cubes $Q$ of side length $1$.  By Theorem \ref{Wongkew}, the number of cubes $Q$ is $\lesssim D L^2$.  
Each tube $T \in \TT'$ intersects $\gtrsim L$ cubes.  So on average, each cube intersects at least $\beta D^{-1}$ tubes.  

Consider triples $(Q, T_1, T_2)$ with $Q$ in our set of cubes and $T_1, T_2 \in \TT'$.  Assuming that $\beta$ is significantly larger than $D$, a Cauchy-Schwarz argument implies that the number of triples is at least $(\beta/D)^2 D L^2 = \beta^2 D^{-1} L^2$.

We group the triples in dyadic blocks according to the size of $\Angle ( v(T_1), v(T_2))$.  If $T_1 \not= T_2$, then this angle is between $1/L$ and $\pi / 2$, so we get $\sim \log L$ dyadic blocks.  We pick a popular dyadic block with angle range $[\theta, 2 \theta]$, where $1/L \le \theta \le 2$.  The number of triples with $\Angle (v(T_1), v(T_2)) \in [\theta, 2 \theta]$ is $\gtrsim \beta^2 D^{-1} L^2 ( \log L)^{-1}. $

There are $\beta L$ tubes in $\TT'$.  By the pigeonhole principle, one of these tubes $T_1$ must appear in $\gtrsim \beta D^{-1} L (\log L)^{-1}$ triples with $\Angle (v(T_1), v(T_2)) \in [\theta, 2 \theta]$.  
We let $H$ (for hairbrush) denote the union of all these tubes, intersected with the ball $(3/2) B$.

Given the angle condition $\Angle(v_1(T), v_2(T)) \sim \theta$, each pair $T_1, T_2$ can appear in $\lesssim \theta^{-1}$ triples, and so the number of tubes $T_2$ in the hairbrush obeys

$$(\# \textrm{ of tubes } T_2 \textrm{ in } H) \gtrsim \beta D^{-1} \theta L (\log L)^{-1}. $$

We will get a lower bound on the volume of $H$ from Wolff's hairbrush argument, and we will get an upper bound on the volume of $H$ from Wongkew's theorem.  Playing these bounds against each other, we will get the desired upper bound $\beta \lesssim D^2 \log^2 L$.  

The tubes in the hairbrush $H$ are morally disjoint.  We can divide the hairbrush into $\sim \theta L$ planar slabs of thickness 1.  Outside of the $(\theta/10)L$-neighborhood of the core line of $T_1$, any point lies in $\lesssim 1$ of the planar slabs.  Because of the angle condition, the tubes in each planar slab have angle separation $\gtrsim L^{-1}$.  By a standard argument, the volume of their union is at least $(\log L)^{-1}$ times the sum of their volumes.  (See for example Theorem 1.3 in Lecture Notes 6 of \cite{Tnotes}.)  Therefore, we see that

$$ \Vol H \gtrsim (\log L)^{-1} (\# \textrm{ of tubes } T_2 \textrm{ in } H) L \gtrsim \beta D^{-1} \theta L^2 (\log L)^{-2}. $$

On the other hand, the hairbrush $H$ is contained in a cylinder around the core line of $T_1$ with radius $\theta L$.  This cylinder is approximately a rectangle of dimensions $\theta L \times \theta L \times L$.  The hairbrush $H$ is contained in the $O(1)$-neighborhood of $Z(P)$ inside this rectangle.  Theorem \ref{Wongkew'} gives the following upper bound on the volume of $H$.  

$$ \Vol H \lesssim D \theta L^2. $$

Combining the last two inequalities, we see that $\beta \lesssim D^2 (\log L)^2$.  

\end{proof}

\end{document}